\documentclass[12pt,reqno]{amsart}
\usepackage[margin=1.0in]{geometry} 
\usepackage{amscd, amssymb, amsmath,amsfonts}
\usepackage{mathrsfs}
\usepackage{bbm}
\usepackage{graphicx}
\graphicspath{{./Figures/}}
\usepackage{subcaption}
\usepackage{tikz-cd}
\usepackage[title]{appendix}
\usepackage[hidelinks]{hyperref}
\usepackage{thmtools,thm-restate}
\usepackage{todonotes}

\newtheorem{thm}{Theorem}[section]
\newtheorem{lem}[thm]{Lemma}
\newtheorem{cor}[thm]{Corollary}
\newtheorem{prop}[thm]{Proposition}

\theoremstyle{definition}
\newtheorem{defn}[thm]{Definition}

\newtheorem{rem}[thm]{Remark}
\newtheorem{ex}[thm]{Example}

\newcommand{\B}{\mathscr{B}}
\newcommand{\C}{\mathbb{C}}
\newcommand{\Cx}{\C^\times}
\newcommand{\Q}{\mathbb{Q}}
\newcommand{\Z}{\mathbb{Z}}
\newcommand{\g}{\mathfrak{g}}
\newcommand{\Uqg}{U_q(\g)}
\newcommand{\Urg}{\overline{U}_q(\g)}

\newcommand{\Uzg}{\overline{U}_\zeta(\g)}
\newcommand{\U}{\overline{U}}
\newcommand{\sltwo}{\mathfrak{sl_2}}

\newcommand{\slthree}{\mathfrak{sl}_3}
\newcommand{\Uzslthree}{\overline{U}_\zeta(\slthree)}
\newcommand{\bracks}[1]{\langle #1\rangle}
\newcommand{\comm}[1]{}
\newcommand{\floor}[1]{\left\lfloor #1 \right\rfloor}
\newcommand{\Deltasl}{{\Delta\mkern-12mu/}}
\newcommand{\R}{\mathcal{R}}
\newcommand{\X}{\mathcal{X}}
\newcommand{\mvec}[1]{{\smash{\text{\boldmath{$#1$}}}}}
\newcommand{\hotimes}{\widehat{\otimes}}
\newcommand{\s}{\mvec{s}}
\newcommand{\Vt}{V(\mvec{t})}
\newcommand{\Vs}{V(\mvec{s})}
\newcommand{\Vts}{V(\mvec{t})\hotimes V(\mvec{s})}
\newcommand{\A}{\mathcal{A}}
\newcommand{\Phip}{{\Phi^+}}
\newcommand{\Deltap}{{\Delta^+}}
\newcommand{\Hom}{\text{Hom}}

\newcommand{\x}{\mvec{\mu_1}}
\newcommand{\y}{\mvec{\mu_2}}
\renewcommand{\H}{\X_{12}}
\renewcommand{\t}{\mvec{t}}
\renewcommand{\P}{\mathcal{P}}
\renewcommand{\DH}{\widehat{\Delta}(\H)}
\begin{document}
	\title[Verma Modules for Restricted Quantum Groups at a Fourth Root of Unity]{Verma Modules for Restricted Quantum Groups at a Fourth Root of Unity}
	\author{Matthew Harper}
	\begin{abstract} \noindent
		For a semisimple Lie algebra $\mathfrak{g}$ of rank $n$, let $\overline{U}_\zeta(\mathfrak{g})$ be the restricted quantum group of $\mathfrak{g}$ at a primitive fourth root of unity. This quantum group admits a natural Borel-induced representation $V({\boldsymbol{t}})$, with ${\boldsymbol{t}}\in(\mathbb{C}^\times)^n$ determined by a character on the Cartan subalgebra. Ohtsuki showed that for $\mathfrak{g}=\mathfrak{sl}_2$, the braid group representation determined by tensor powers of $V({\boldsymbol{t}})$ is the exterior algebra of the Burau representation. In this paper, we generalize the tensor decomposition of $V({\boldsymbol{t}})\otimes V({\boldsymbol{s}})$ used in Ohtsuki's proof to any semisimple $\mathfrak{g}$. Upon specializing to the $\slthree$ case, we describe all projective covers of $V({\boldsymbol{t}})$ in terms of induced representations. The above decomposition formula for $V({\boldsymbol{t}})\otimes V({\boldsymbol{s}})$ is then extended to more general $\boldsymbol{t}$ and $\boldsymbol{s}$ where these projective covers occur as indecomposable summands.  We also define a stratification of $(\mathbb{C}^\times)^{4}$ whose points $({\boldsymbol{t}},{\boldsymbol{s}})$ in the lower strata are associated with representations $V({\boldsymbol{t}})\otimes V({\boldsymbol{s}})$ that do not have a homogeneous cyclic generator. With this information, we characterize under what conditions the isomorphism ${V({\boldsymbol{t}})\otimes V({\boldsymbol{s}})\cong V({\boldsymbol{\lambda t}})\otimes V({\boldsymbol{\lambda^{-1} s}})}$ holds.
	\end{abstract}
	\maketitle
	\setcounter{tocdepth}{1}
	\tableofcontents
	\section{Introduction}
	The study of quantum groups at fourth roots of unity is motivated both by central questions in quantum topology and universal behaviors in the representation theory of quantum groups. Regarding the former, exact relations between quantum invariants of knots and 3-manifolds and classical topological and geometrically constructed invariants have been an important focus of research in quantum topology since its inception over three decades ago. Both the Alexander-Conway polynomial of knots and the Reidemeister torsion of 3-manifolds are obtained by quantum invariants derived from low-rank quantum groups, including quantum $\sltwo$ at a fourth root of unity. See \cite{KauffmanSaleur,KohliPatureau,Sartori,Viro} for identifications of the Alexander polynomial with quantum invariants derived from $R$-matrices. A modification of the Reshetikhin-Turaev invariant for $\overline{U}_\zeta^H(\sltwo)$, the unrolled restricted quantum group at a fourth root of unity, is shown to equal classical torsion in \cite{BCGP} via the Turaev surgery formula \cite{TuraevReidemeister}. We expect that other classical topological invariants can be recovered from quantum groups at a fourth root of unity.\newline 
	
	The relation to the Alexander polynomial has been extended to the Burau representation of the braid group by Jun Murakami \cite{Murakami, MurakamiStateModel} in the context of state sum models. Ohtsuki \cite{Ohtsuki} describes this relation in terms of an explicit isomorphism between two braid group representations. One is the Turaev-type \cite{Turaev88} $R$-matrix action on tensor powers of a two-dimensional Verma module $V(t)$ of the restricted quantum group $\overline{U}_\zeta(\sltwo)$, where $t\in \Cx$ denotes the highest weight. The other is the full exterior algebra of the unreduced Burau representation. We may view this action on the exterior algebra as a twisted action of the braid group on the cohomology ring of the $U(1)$-character variety of the $n$-punctured disk.\newline 
	
	A natural generalization is to understand $R$-matrix representations from $\overline{U}^H_\zeta(\g)$ for higher rank $\g$ in similar classical topological terms. Identifying tensors of basis elements with geometric generators as in \cite{Ohtsuki} requires explicit descriptions of both Verma modules with continuous highest weights and their tensor decompositions. These give insights into underlying ring structures and spectral properties of braid generators via skein relations.\newline 
	
	While quantum invariants in rank one have been studied extensively in the literature, the respective algebraic and topological constructions are more complicated and have received less attention for higher rank Lie types, and little is known in the way of descriptions. This article is an introduction to studying the higher rank restricted quantum groups.\newline 
	
	A second rationale for focusing on the fourth root of unity is that this case captures much of the behavior of the representation theory of quantum groups at general roots of unity. One may consider the blocks of the abelian representation category of a quantum algebra, obtained from minimal central idempotents of the algebra. It has been known for some time that all blocks constituting the representation category of quantum $\sltwo$ at general roots of unity are isomorphic to those at a fourth root of unity, see \cite{CGP,Kerler}. Among odd orders, $p$ of roots of unity and general Lie types, Andersen, Jantzen, and Soergel \cite{AJS} show that blocks are independent of $p$. Therefore, it is reasonable to expect that a similar result holds for even roots of unity. More specifically, block types, decomposition series of indecomposable representations, and descriptions of projective covers studied here should be mostly unchanged for other even roots of unity.\newline 
	
	This category of modules will, however, depend on the order of the root of unity when considered as a tensor category. Even so, our exploration of the fourth root of unity case reveals general phenomena likely to be encountered for arbitrary roots. The setting in this paper produces a multi-parameter fusion ring very different from the classical integral theory with irreducibility decomposition rules determined by generic and singular loci of the parameters. We identify the modules which occur as particular induced representations depending on the stratum of degeneracy, hinting towards an interesting, more general theory for all Lie types.
   
	\subsection{Statement of Results} We start with basic definitions and properties of $\Uzg$ and its induced modules for a semisimple Lie algebra of rank $n$. We then give the generic tensor decomposition for its representations $\Vt$. However, the goal of this paper is to give a description of interesting structures which arise in the $\slthree$ case in the non-semisimple setting that can be further expanded upon in higher rank.\newline
	
	We recall, in Section \ref{sec:notation}, the construction of quantum groups at roots of unity from Lusztig's divided powers algebra $U_q^{div}(\g)$. In contrast to the small quantum group $u_q(\g)$ described in \cite{LusztigADE,Lusztig}, the restricted quantum group $\Urg$ considered here is infinite dimensional. For $q=\zeta$, a primitive fourth root of unity, we give a generators and relations description and a PBW basis of $\Uzg$ for each Lie type. At this root of unity,  $E_\alpha^2=F_\alpha^2=0$ for every root $\alpha$, and the Serre relations are reduced to ``far commutativity.''\newline 
	
	Outside of the simply laced types, some $E_\alpha$ and $F_\alpha$ are zero. We refer to these roots $\alpha$ as \textit{negligible}. Here, $\overline{\Phip}$ denotes the set of positive roots which are not negligible, equipped with an ordering $<_{br}$, and $\Psi$ is the collection of maps $\overline{\Phip}\to\{0,1\}$.\newline

	Consider the group of characters $\P$ on the Cartan torus of the restricted quantum group. A character $\t$ is determined by the images $t_i\in \C^\times$ of Cartan generators $K_i$ for $1\leq i\leq n$. Thus, $\t$ can be identified with a tuple $(t_1,\dots, t_n)\in (\Cx)^n$. Multiplication in $\P$ is given by $\t\s=(t_1s_1,\dots, t_ns_n)$ with identity $\mvec{1}=(1,\dots, 1)$. The character $\t$ extends to a character $\gamma_{\t}$ on the Borel subalgebra $B$ by taking the value zero off the Cartan torus.\newline
	
	Let $V_{\t}=\langle v_h\rangle$ be the 1-dimensional left $B$-module determined by $\gamma_{\t}$, i.e. for $b\in B$, ${bv_h=\gamma_{\t}(b)v_h}$.
	We then define the representation $V(\t)$ to be the induced module 
	\begin{align}
	\Vt={\text{{Ind}} _{B}^{\Urg}(V_{\t})=\Urg\otimes_{B}V_{\t}}. 
	\end{align}
		We fix $q=\zeta$, a primitive fourth root of unity. In this case, the representations $\Vt$ each have dimension $|\Psi|$ which, in the simply-laced case, is equal to $2^{|\Phip|}$. 
	\begin{lem}
		The representations $\Vt$ of $\Uzg$ are indecomposable and non-isomorphic for each $\t\in\P$. Moreover, there exists an algebraic set $\R$ so that $\Vt$ is irreducible if and only if $\t\in\P\setminus\R$.
	\end{lem}

	Let $\mvec{\sigma^\psi}$ denote the weight of $F^\psi v_h$ in $V(\mvec{1})$. A pair of characters $(\t,\s)\in\P^2$ is called \emph{non-degenerate} if $V(\mvec{\sigma^\psi}\t\s)$ is irreducible for each $\psi\in\Psi$.  Our first main result is a decomposition rule for  $\Vt\otimes\Vs$, given that the pair $(\t,\s)$ is non-degenerate.\\
	
	\begin{restatable}[Semisimple Tensor Product Decomposition]{thm}{directsum}
		\label{thm:directsum}
		Let $\g$ be a semisimple Lie algebra. If $(\t,\s)$ is a non-degenerate pair, then the tensor product $\Vt\otimes\Vs$ decomposes as a direct sum of irreducibles according to the formula
		\begin{align}
		\Vt\otimes\Vs\cong\bigoplus_{\psi\in\Psi}V(\mvec{\sigma^\psi}\t\s).
		\end{align}
	\end{restatable}

	Theorem \ref{thm:directsum} generalizes the generic tensor product formula given in \cite{Ohtsuki} for the $\sltwo$ case to any semisimple Lie type:
	\begin{align}
	V(t)\otimes V(s)\cong V(ts)\oplus V(-ts).
	\end{align} 
	
	\comm{Observe that the tensor product of reducible representations may still be isomorphic to a direct sum of indecomposables. For example, consider $V(t,1)\otimes V(1,s)$ with $t$ and $s$ generic.\newline }
	
	We may also consider tensor products for which $(\t,\s)$ is a degenerate pair by giving decompositions in terms of the projective covers of $\Vt$. 
	\begin{restatable}{lem}{projgen}\label{lem:projgen}
		The representation $\Vt$ is projective in $\U$-mod if and only if $\t\notin \R$.
	\end{restatable}

	We give a description of this decomposition into reducible indecomposables for the $\slthree$ case. Let $B^\emptyset=B$ and for each nonempty $I\subsetneq \overline{\Phip}$ we consider a subalgebra $B^I$ of $\U$. For example:
	\begin{align}
	B^{1}=\bracks{K_1,K_2,E_{12},E_2},&& \text{and} &&
	B^{1,2}=\bracks{K_1,K_2,F_1E_2E_1,F_2E_1E_2}.
	\end{align}

	Let $\bracks{p_0^I}$ be the $B^I$-module on which the Cartan generators $K_i$ act by $t_i$ and all other generators of $B^I$ act trivially. Denote the induced representation $\text{Ind}_{B^I}^{\U}(\bracks{p_0^I})$ by $P^I(\t)$. \newline 
	
	Consider the subsets of $\P$:
	\begin{align}
		\X_1&=\{\t\in\P :t_1^2=1\},  & 
		\X_2&=\{\t\in\P :t_2^2=1\}, & 
		\H&=\{\t\in\P :(t_1t_2)^{2}=-1\},
	\end{align}
	and $\R=\X_1\cup\X_2\cup\H.$
We partition $\R$ into disjoint subsets indexed by nonempty subsets $I\subsetneq \overline{\Phip}$, with
	\begin{align}
		\R_{I}=\left(\bigcap_{\alpha\in I} \X_\alpha\right)\setminus \left(\bigcup_{\alpha\notin I}\X_\alpha\right).
	\end{align} We define $\R_\emptyset$ to be $\P\setminus\R$, so that  $\{\R_I\}_{I\subsetneq\overline{\Phip}}$ yields a partition of $\P$.\\

	\begin{restatable}{thm}{projind}
		For each $I\subsetneq \overline{\Phip}$ and $\t\in\R_I$, 
		$P^I(\t)$ is the projective cover of $\Vt$.\\
	\end{restatable}

	We prove that each $P^I(\t)$ is hollow and indecomposable for all $\t\in\R_I$. For each  subset $I\subsetneq \overline{\Phip}$ we define $\Psi_I(\t)\subseteq \Psi$ and a collection of subsets of roots $\mathscr{I}(I)$, that together index the projective representations and their position in tensor products. We generalize the generic tensor product decomposition for the $\slthree$ case.\\
	
	\begin{restatable}{thm}{projdecomp}\label{thm:projdecomp}
		Let $\g=\slthree$. Suppose $\t,\s\notin\R$ and $\t\s\in\R_J$. Then $V(\t)\otimes V(\s)$ decomposes as a direct sum of indecomposable representations:
		\begin{align}
		V(\t)\otimes V(\s)\cong \bigoplus_{I\in \mathscr{I}(J)}\bigoplus_{\psi\in \Psi_I(\t\s)}P^I(\mvec{\sigma^\psi\t\s}).
		\end{align}
	\end{restatable}

	We do not consider the tensor product between two projective covers in this paper. However, since these are induced representations, we expect that the methods used in the proof of Theorem \ref{thm:directsum} can be applied.\\ 
	
	Observe that these tensor product representations only depend on the product of $\t$ and $\s$, thus motivating the following definition. We call an isomorphism  
	\begin{align}
	\Vt\otimes\Vs\cong V(\mvec{\lambda}\t)\otimes V(\mvec{\lambda^{-1}}\s),
	\end{align} 
	for some $\mvec{\lambda}\in\P$, a \emph{transfer}.\newline

	In order to classify representations not considered in Theorem \ref{thm:projdecomp}, we give a description of all transfers on tensor product representations ${\Vt\otimes\Vs}$ of $\Uzslthree$. Our approach is to first find representations generated by a single weight vector under the action of non-Cartan elements, we call such a representation \emph{homogeneous cyclic}. If $\Vt\otimes\Vs$ is homogeneous cyclic then it is characterized by the weight $\mvec{-ts}$ of its generator. The values of $(\t,\s)$ for which cyclicity fails determine the \emph{acyclicity locus} $\A$. \\
	
	\begin{restatable}[Homogeneous Cyclic Tensor Product Representations]{thm}{cyclic} \label{thm:cyclic}
		The acyclicity locus $\A\subseteq\P^2$ is given by
		\begin{align}
		\X_1^2\cup\X_2^2\cup\H^2\cup(\H\times\R_{1,2})
		\cup(\R_{1,2}\times\H).
		\end{align}
	\end{restatable}
	Let
	\begin{align}
	\DH&=\{(\t,\s)\in \H^2:(t_1s_1)^2=1\}\\
	&=\{(\t,\s)\in \P^2:(t_1s_1)^2=1, (t_1t_2)^2=(s_1s_2)^2=-1\}.
	\end{align}
	Then $\P^2$ is stratified according to the filtration $\P^2_0\subset\P^2_1=\A\subset\P^2_2=\P^2$,
	with
	\begin{equation}
	\mathcal{P}^{2}_0=\R_{1,2}^2\cup\R_{1,12}^2\cup\R_{12,2}^2\cup \DH\cup((\R_{1,12}\cup\R_{12,2})\times \R_{1,2})\cup( \R_{1,2}\times (\R_{1,12}\cup\R_{12,2}))
	\end{equation}
	\comm{ We illustrate the inclusion $\P_0^2\subseteq\P_1^2$ in Figure \ref{fig:stratify}	below. 
	
	{	\begin{figure}[h]\centering
			\includegraphics[]{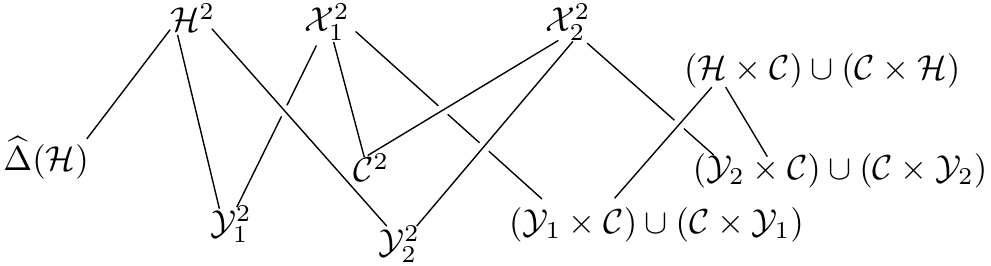}\caption{The inclusion of $\mathcal{P}^2_0$ in $\mathcal{P}^2_1$.  }\label{fig:stratify}
	\end{figure}}}
	
	We may also define an action of $\P$ on $\P^2$ as follows. Let $\mvec{\lambda}\in \P$ and $(\t,\s)\in \P^2$, then  
	$
	\mvec{\lambda}\cdot(\t,\s)=(\mvec{\lambda t},\mvec{\lambda^{-1}s}).
	$
	The swapping of coordinates is replicated by the action of $\mvec{\lambda}=\mvec{t^{-1}s}$. As $\P^2_0$ and $\P^2_1$ are preserved under the exchange of $\t$ and $\s$, the stratification respects the equivalence determined by a braiding. We group the defining subsets of $\P^2_0$ and $\P^2_1$ so that they are preserved by swaps, and we refer to the resulting subsets as \emph{symmetrized}. Under this grouping, we have the following theorem.\\
	
	\begin{restatable}[Transfer Principle]{thm}{transfer}\label{thm:transfer}
		Suppose $(\t,\s)$ belongs to a symmetrized subset in the $n$-stratum. Then 
		\begin{align}
		\Vt\otimes\Vs\cong V(\mvec{\lambda}\t)\otimes V(\mvec{\lambda^{-1}}\s)
		\end{align}
		if and only if $\mvec{\lambda}\cdot(\t,\s)$ belongs to the same symmetrized subset in the $n$-stratum.\\
	\end{restatable}	
	
	\subsection{Future Work}\label{sec:future} 
	We have found that the subalgebras we induce over to construct the projective covers of $\Vt$ have a short list of generators given by Cartan elements and products of root vectors. Such a description should lend itself to an underlying theory that allows us to easily describe projective covers for any $\g$, indexed by subsets of positive root vectors. \newline
	
	Although Theorem \ref{thm:transfer} addresses isomorphism classes of tensor products, we do not give explicit tensor decompositions when both $\t,\s\in\R$ and $(\t,\s)$ is degenerate. The tensor product of $P^I(\t)$ with $P^J(\s)$ is not discussed here since there are many combinations to consider. We leave these cases for a full generalization of the theory. \newline
		
	We defined the representations $\Vt$ for any restricted quantum group, but we have only considered the fourth root of unity case. As explained above, we expect the representation category of $\overline{U}_q(\slthree)$ at other roots of unity, and those of higher rank quantum groups, to be comparable with what we have found here.  \newline 
	
	In \cite{Harpersl3}, we give the various decompositions of tensor products between irreducible subrepresentations of $\Vt$ and $\Vs$. We also study the link invariant obtained from quantum group representations $\Vt$. In the $\slthree$ case, we prove that certain evaluations of  this invariant yield the Alexander polynomial of knots. For this reason, we expect that the Burau representation and its geometric interpretation naturally include into the higher rank constructions associated with the quantum $\slthree$ invariant.

	\subsection{Structure of Paper}  In Section \ref{sec:notation} we recall the quantum group $U_q(\g)$ according to Lusztig \cite{Lusztig} and define the restricted quantum group $\overline{U}_q(\g)$. We prove various properties of these algebras when $q=\zeta$. In Section \ref{sec:Vt}, we define the induced representations $\Vt$ for any root of unity before assuming $q=\zeta$. We then prove indecomposibility and irreducibility properties of these representations. The results on irreducibility are then applied in Section \ref{sec:directsum} to find a direct sum decomposition for sufficiently generic tensor product representations.\newline 
	
	 In Section \ref{sec:props} we specialize the above results and give explicit examples in the $\slthree$ case. We then discuss the projective covers of $V(\t)$ for each $\t\in\P$ in Section \ref{sec:proj}, and give the various decompositions of $\Vts$ in terms of these representations.  \newline
	
	Sections \ref{sec:cyclicgen}, \ref{sec:cyclic1}, and \ref{sec:cyclici} are concerned with finding homogeneous cyclic representations, and transfer isomorphisms which are not implied by Theorem \ref{thm:projdecomp}. Section \ref{sec:cyclicgen} sets up the language and the method used for finding cyclic representations,  we also characterize cyclicity in the generic case. Sections \ref{sec:cyclic1} and \ref{sec:cyclici} each study cyclicity for some non-generic choice of characters. Gathering the conclusions of these sections yields the cyclicity theorem and transfer principle for representations $\Vt\otimes\Vs$, stated in Section \ref{sec:thms}.\newline
	
	General computations which are referenced throughout the paper are compiled in \hbox{Appendix \ref{sec:comp}}. The latter section of the Appendix includes information on induced representations used in proving Theorem \ref{thm:directsum}.

	\subsection{Acknowledgments}
	I am very grateful to Thomas Kerler for helpful discussions. I would also like to extend my thanks to Sachin Gautam for useful suggestions. I thank the NSF for support through the grant \textbf{NSF-RTG \#DMS-1547357}.

	\section{Restricted Quantum Groups}\label{sec:notation}
	In this section, we recall the definition of the quantum group for a semisimple Lie algebra $\g$. Following Lusztig \cite{LusztigADE,Lusztig}, we obtain a restricted quantum group by setting the deformation parameter $q$ to a root of unity in the context of the divided powers algebra. We call the restricted quantum group at a primitive fourth root of unity $\Uzg$. In the non-simply laced cases, we relate the restricted quantum groups to those of type $A$. We give a generators and relations description, and a PBW basis of $\Uzg$. We then prove various properties of these algebras, which will be useful in later sections.
	\newline 
	
	\subsection{Lusztig's Construction} Let $q$ be a formal parameter and let $\g$ be Lie algebra with $n\times n$ Cartan matrix $(A_{ij})$ symmetrized by the vector $(d_i)$ with entries in $\{1,2,3\}$. Let $\Phip$ be the space of positive root vectors, and $\Delta^+$ the positive simple roots of $\g$. The positive simple roots are indexed so that $(\alpha_i,\alpha_j)=A_{ij}$
	We define the quantum group $U_q(\g)$ following \cite{LusztigADE,Lusztig} and refer the reader there for additional details. We set
	\begin{align}
	[N]_d!=\prod_{j=1}^N\frac{q^{dj}-q^{-dj}}{q^d-q^{-d}}, &&
	\left[\!\begin{matrix} M+N \\ N \end{matrix}\!\right]_{d}
	=\frac{[M+N]_d!}{[M]_d![N]_d!},
	\end{align} and \begin{align}\label{eq:floor}
	\floor{x}_d=\frac{x-x^{-1}}{q^d-q^{-d}},
	\end{align} 
	omitting subscripts when $d=1$. 
	
	\begin{defn}\label{def:qgrp}
		Let $\Uqg$ be the algebra over $\Q(q)$ generated by $E_i$, $F_i$, and $K^{\pm1}_i$ for ${1\leq i\leq n}$ subject to the relations:
		\begin{align}
		&K_iK_i^{-1}=1, & &K_iK_j=K_jK_i,\label{eqn:qgroup1}\\
		&K_iE_j=q^{d_iA_{ij}}E_jK_i, & &K_iF_j=q^{-d_iA_{ij}}F_jK_i,\label{eqn:qgroup2}
		\end{align}	
		\begin{align}\label{eqn:qgroup3}
		[E_i,F_j]=\delta_{ij}\floor{K_i}_{d_i},
		\end{align}
		\begin{align}\label{eq:SerreE}
		&\sum_{r+s=1-A_{ij}}(-1)^s\bigl[\!\begin{smallmatrix} 1-A_{ij} \\ s \end{smallmatrix}\!\bigr]_{d_i} E_i^rE_jE_i^s=0, & \text{for }i\neq j,
		\end{align}	
		\begin{align}\label{eq:SerreF}
		&\sum_{r+s=1-A_{ij}}(-1)^s\bigl[\!\begin{smallmatrix} 1-A_{ij} \\ s \end{smallmatrix}\!\bigr]_{d_i} F_i^rF_jF_i^s=0, & \text{for }i\neq j.
		\end{align}
	\end{defn}
	Equations (\ref{eq:SerreE}) and (\ref{eq:SerreF}) are called the \emph{quantum Serre relations}.
	
	\begin{defn}
		Let $U_q^{div}(\g)$ be the subalgebra of $U_q(\mathfrak{\g})$ generated by 
		\[
		E_i^{(N)}=\frac{E_i^N}{[N]_{d_i}!},F_i^{(N)}=\frac{F_i^N}{[N]_{d_i}!},K_i^{\pm}
		\] 
		 for $N\geq0$ over $\Z[q,q^{-1}]$. We call $U_q^{div}(\g)$ the \emph{divided-powers algebra}.
	\end{defn}
	
	The Hopf algebra structure on $\Uqg$ is defined by the maps below for $1\leq i \leq n$, and extends to the entire algebra via their (anti-)homomorphism properties:
	\begin{align}
	\Delta(E_i)&=E_i\otimes K_i+1\otimes E_i & S(E_i)&=-E_iK_i^{-1}&\epsilon(E_i)&=0\label{eq:HopfE}\\
	\Delta(F_i)&=F_i\otimes 1+ K_i^{-1}\otimes F_i & S(F_i)&=-K_iF_i&\epsilon(F_i)&=0\label{eq:HopfF}\\
	\Delta(K_i)&=K_i\otimes K_i&S(K_i)&=K_i^{-1}&\epsilon(K_i)&=1\label{eq:HopfK}.
	\end{align}
	According to \cite{LusztigADE,Lusztig}, powers of $K_i$ and the collection of maps $\psi:\Phip\rightarrow\mathbb{Z}_{\geq0}$, together with the braid group action $T_i$ on the quantum group defined therein determine a PBW basis of $\Uqg$.\\
	
	Let $\Q_l$ be the quotient of $\Q[q,q^{-1}]$ by the ideal generated by the $l$-th cyclotomic polynomial. Let $l_i$ be the order of $q^{2d_i}$ in $\Q_l$. For each $\alpha\in\Phip$ set $l_\alpha=l_i$ if $\alpha$ is in the Weyl orbit of $\alpha_i\in \Deltap$. The divided powers elements $E_\alpha^{(N)}$ and $F_\alpha^{(N)}$ for $\alpha\in\Phip$ and $0\leq N\leq l_\alpha-1$, together with $K_i^\pm$ for $1\leq i\leq n$ generate a $\Q_l$-subalgebra $\overline{U}_q^{div}(\g)\subseteq U_q^{div}(\g)\otimes_{\Z[q,q^{-1}]} \Q_l$. As elements of ${U}_q^{div}(\g)\otimes_{\Z[q,q^{-1}]} \Q_l$,
	\begin{align}
	E^{l_\alpha}_\alpha=[l_\alpha]_{d_\alpha}!E^{(l_\alpha)}_\alpha=0
	&& \text{and}&& 
	F^{l_\alpha}_\alpha=[l_\alpha]_{d_\alpha}!F^{(l_\alpha)}_\alpha=0.
	\end{align}
	
	\begin{defn}
		The \emph{restricted quantum group} $\Urg$ is the $\Q_l$-algebra generated by $E_i,F_i,$ and $K_i^\pm$ inside $\overline{U}_q^{div}(\g)$. The Hopf algebra structure on $\Urg$ is inherited from the one carried by $\Uqg$.
	\end{defn}
	
	\begin{rem}
		In contrast to the small quantum group $u_q(\g)$ described in \cite{LusztigADE,Lusztig}, the restricted quantum group $\Urg$ considered here is infinite dimensional.
	\end{rem}

	\subsection{Specialization of $q$} Recall that $l_i$ is the order of $q^{2d_i}$ in $\Q_l$. For the remainder of this section, assume $q=\zeta$ so that $l=4$. If $d_i=2$, then $l_i=1$ and implies $E_i=0$. If $d_i\in\{1,3\}$, then $l_i=2$ and $E^{2}_i=F_i^2=0$.\newline 
		
	These relations extend to non-simple root vectors using the braid group action. Note that at a fourth root of unity, the Serre relations found in equations (\ref{eq:SerreE}) and (\ref{eq:SerreF}) reduce to ``far commutativity.'' However, if $\text{rank}(\g)>2$, additional commutation relations must be considered so that the maps $T_i$, used for defining higher root vectors, are automorphisms of the restricted quantum group. 
	
	\begin{defn}
		Suppose $E_\alpha$, or $F_\alpha$, in $U_\zeta^{div}(\g)$ does not belong to the subalgebra over $\Z[\zeta]$ generated by simple root vectors or that $d_i=2$. Then the root $\alpha$ is said to be \emph{negligible}. The collection of negligible positive roots is denoted by $\Phi_{0}^+\subseteq\Phip$. We set $\Delta_{0}^+=\Deltap\cap \Phi_{0}^+$, $\overline{\Phip}=\Phip\setminus \Phi_{0}^+$, and $\overline{\Deltap}=\Deltap\setminus\Delta_{0}^+$.  In addition, $\overline{\Phip}$ is equipped with some ordering $<_{br}$ according the braid group action mentioned above.
	\end{defn}

	We refer the reader to \cite{CP} for more details on $<_{br}$. In the $\slthree$ case, we use the ordering 
	\begin{align}
	\alpha_1<_{br}\alpha_1+\alpha_2<_{br}\alpha_2.
	\end{align} The root vectors associated with the root $\alpha_1+\alpha_2$ are
	\begin{align}
	E_{12}=-(E_1E_2+\zeta E_2E_1)  &&\text{and}&& F_{12}=-(F_2F_1-\zeta F_1F_2).
	\end{align}

	\begin{rem}
		Upon deletion of negligible simple roots in types $BCF$, those with \hbox{$d_i=2$} in the convention of \cite{Bourbaki4-6}, we can make identifications in the resulting Cartan data. This leads to the following isomorphisms:
		\begin{align}
		\overline{U}_\zeta(\mathfrak{b}_n)&\cong \overline{U}_\zeta(\mathfrak{a}_{n-1})[K_n^\pm]/\bracks{\{K_n,E_{n-1}\},\{K_n,F_{n-1}\}}
		\\
		\overline{U}_\zeta(\mathfrak{c}_n)&\cong \overline{U}_\zeta(\mathfrak{a}_{1})[K_2^\pm,\dots,K_n^\pm]/\bracks{\{K_2,E\},\{K_2,E\}}
		\\
		\overline{U}_\zeta(\mathfrak{f}_4)&\cong
		\overline{U}_\zeta(\mathfrak{a}_2)[K_3^\pm,K_4^\pm]/\bracks{\{K_3,E_2\},\{K_3,F_2\}}.
		\end{align}
		Here we have used $\{\cdot,\cdot\}$ to denote the anticommutator.
	\end{rem}
	
	\begin{rem}
		In type $G_2$, the only non-simple root vector generated by $E_1$, $E_2$ over $\Z[\zeta]$ is $E_{12}=\zeta E_1E_2-E_2E_1$. All other $E_\alpha$ are given by expressions in higher divided powers and so they do not belong to the restricted quantum group. Therefore, we may identify the non-negligible roots in types $\mathfrak{a}_2$ and $\mathfrak{g}_2$. However, the associated quantum groups are not isomorphic because the off-diagonal entries of the symmetrized Cartan matrices, modulo 4, are not equal.
	\end{rem}
	By the above remarks, it is enough to describe the non-Cartan components of the restricted quantum group for the simply-laced cases. We have referred to \cite{LusztigADE} for the following generators and relations description. For a root $\alpha=\sum c^\alpha_j\alpha_j$, we define $g(\alpha)$ to be the greatest index $i$ for which $c_i\neq0$. Let $h(\alpha)=\sum c_j$ and $h'(\alpha)=c_i^{-1}h(\alpha)$.
\begin{restatable}{defnprop}{genrel}
	The restricted quantum group $\overline{U}_\zeta(\g)$ is the $\Q_4$-algebra generated by $E_{\alpha}, F_{\alpha},$  for $\alpha\in \overline{\Phip}$, and $K_j$ for $1\leq j\leq n$ with relations:
	\begin{gather}
		\begin{align}
			&K_iK_i^{-1}=1, && K_iK_j=K_jK_i,\\	
			&K_iE_{\alpha_j}=\zeta^{d_iA_{ij}}E_{\alpha_j}K_i, && K_iF_{\alpha_j}=\zeta^{d_iA_{ij}}F_{\alpha_j}K_i,\\ &[E_{\alpha_i},F_{\alpha_j}]=\delta_{ij}\floor{K_i}_{d_i}, && E_\alpha^2=F_\alpha^2=0,
		\end{align}\\
		\left.\begin{matrix}
			E_\alpha E_{\alpha_i}=E_{\alpha_i}E_\alpha\\
			F_\alpha F_{\alpha_i}=F_{\alpha_i}F_\alpha
		\end{matrix}\right\rbrace
		\text{ for } (\alpha,\alpha_i)=0, i<g(\alpha), \text{ and } h'(\alpha)\in\Z,\\
		\left.\begin{matrix}
			E_\beta E_{\alpha}=\zeta E_\alpha E_\beta + \zeta E_{\alpha+\beta}\\
			E_\beta E_{\alpha+\beta}= -\zeta E_{\alpha+\beta}E_\beta \\
			E_\alpha E_{\alpha+\beta}= \zeta E_{\alpha+\beta}E_\alpha \\
			F_\alpha F_{\beta}=-\zeta F_\beta F_\alpha - \zeta F_{\alpha+\beta}\\
			F_\beta F_{\alpha+\beta}=-\zeta F_{\alpha+\beta}F_\beta \\
			F_\alpha F_{\alpha+\beta}= \zeta F_{\alpha+\beta}F_\alpha
		\end{matrix}\right\rbrace
		\begin{matrix} \text{ for } (\alpha,\beta)=-1 \text{ and either }\\
			\beta=\alpha_i \text{ and } i<g(\alpha) \text{ or}\\
			h(\beta)=h(\alpha)+1\text{ and } g(\alpha)=g(\beta).
		\end{matrix}
	\end{gather}
\end{restatable}
	
	The Hopf algebra structure on $\Uzg$ is inherited from $\Uqg$, and described in equations (\ref{eq:HopfE})-(\ref{eq:HopfK}). Let $\Psi$ denote the space of maps $\psi:\overline{\Phip}\to\{0,1\}$. We now state a modification of Theorem 8.3 from \cite{Lusztig}.
	\begin{restatable}{cor}{PBW}\label{cor:basis}
		The restricted quantum group $\Uzg$ has a PBW basis
		\begin{align}
			\{E^{\psi}F^{\psi'}K^{\underline{k}}:\psi,\psi'\in\Psi \text{ and } \underline{k}\in \Z^n\},
		\end{align}
		where $E^{\psi}=\prod_{\alpha\in \overline{\Phip}}E_\alpha^{\psi(\alpha) }$, 
		$F^{\psi}=\prod_{\alpha\in \overline{\Phip}}F_\alpha^{\psi(\alpha) }$,  $K^{\underline{k}}=\prod_{i=1}^nK_i^{{k}_i}$, and products are ordered with respect to $<_{br}$. 
	\end{restatable}
		\begin{lem}\label{lem:central}
		Let $\underline{k}\in \Z^n$ and $K^{\underline{k}}=\prod_{i=1}^n K_i^{{k}_i}$ be a basis vector of the Cartan torus. Then $K^{\underline{k}}$ is central if and only if \[(d_iA_{ij})\underline{k}\in (l\Z)^n.\]
	\end{lem}

	\begin{proof} The proof follows from
		$
		K^{\underline{k}}E_j
		=q^{\sum_i \underline{k}_i (d_iA_{ij}) }E_jK^{\underline{k} }
		=q^{ ((d_iA_{ij})\underline{k})_j }E_jK^{\underline{k} }
		$.
	\end{proof}

\begin{lem}\label{lem:Kcomm}
	For each $\alpha=\sum c_i^\alpha\alpha_i\in\overline{\Phip}$, let  $K_\alpha={\prod K_i^{c^\alpha_i}}$. Then $
	[E_\alpha,F_\alpha]=\floor{K_\alpha}$ for all $\alpha\in\overline{\Phip}$.
\end{lem}

\begin{proof}
	We give a proof by induction on the height $h(\alpha)=\sum c_i$. If $h(\alpha)=1$, then $\alpha$ is simple and $[E_i,F_i]=\floor{K_i}$. Suppose now that $[E_\alpha,F_\alpha]=\floor{K_\alpha}$ for some $\alpha\in \overline{\Phip}$, we prove that $[E_{\alpha+j},F_{\alpha+j}]=\floor{K_{\alpha+j}}$ for any simple root $\alpha_j\in \overline{\Phip}$ such that $(\alpha,\alpha_j)=-1$. We compute,
	\[
	[E_{\alpha+j},F_{\alpha+j}]=\comm{
	[-E_{\alpha}E_{j} -\zeta E_{j} E_{\alpha} ,\zeta F_{\alpha}F_{j}-F_{j}F_{\alpha} ]
	\]
	\[
	=}-\zeta[E_\alpha E_{j},F_\alpha F_{j}]+[E_\alpha E_{j},F_{j}F_{\alpha}]+[E_{j} E_{\alpha},F_{\alpha}F_{j}]
	+\zeta[E_{j} E_{\alpha},F_{j}F_{\alpha}].
	\]
	Since $(\alpha,\alpha_i)=-1$, we have 
	\[T_i([E_i,F_{\alpha}])=[-F_iK_i,\zeta F_iF_\alpha-F_\alpha F_i]
	=F_iK_iF_\alpha F_i+\zeta F_iF_\alpha F_iK_i=-\zeta F_iF_\alpha F_i+\zeta F_iF_\alpha F_iK_i=0
	\]
	and $[E_i,F_\alpha]=[E_\alpha,F_i]=0$. This implies,
	\[
	[E_\alpha E_{j},F_\alpha F_{j}]=E_\alpha F_\alpha E_jF_j-F_\alpha E_\alpha F_jE_j=
	\floor{K_\alpha}E_jF_j -F_\alpha E_\alpha \floor{K_j}
	\]
	\[
	[E_j E_{\alpha},F_j F_{\alpha}]= E_jF_jE_\alpha F_\alpha- F_jE_jF_\alpha E_\alpha=
	E_jF_j\floor{K_\alpha}-\floor{K_j}F_\alpha E_\alpha.
	\]
	So
		\[
	[E_{\alpha+j},F_{\alpha+j}]\comm{=[E_\alpha E_{j},F_{j}F_{\alpha}]+[E_{j} E_{\alpha},F_{\alpha}F_{j}]\]\[}=
	E_\alpha\floor{K_j}F_\alpha +F_j\floor{K_\alpha} E_j +E_j\floor{K_\alpha}F_j+F_\alpha\floor{K_j}E_\alpha.
	\]
	Again making use of $(\alpha,\alpha_i)=-1$, $K_\alpha E_i=\zeta^{-1}K_\alpha$ and $K_j E_\alpha=\zeta^{-1}E_\alpha K_j$. Therefore,
	\[
	[E_{\alpha+j},F_{\alpha+j}]=\floor{K_\alpha}\floor{\zeta K_j}+\floor{K_j}\floor{\zeta K_\alpha}\comm{\]
	\[=\frac{1}{(2\zeta)^2}
	\left(\left(K_\alpha-K_\alpha^{-1}\right)(\zeta K_j+\zeta K_j^{-1})+
	(\zeta K_j- K_j^{-1})\left(\zeta K_\alpha+\zeta K_\alpha^{-1}\right)
	\right)
	\]\[}
	=\frac{1}{2\zeta}\left(K_jK_\alpha-K_j^{-1}K_\alpha^{-1}\right)=\floor{K_{\alpha+j}}.
	\qedhere\]
\end{proof}

	\subsection{Properties of $E^\psi$ and $F^\psi$} Consider the subalgebras of $\Uzg$:
	\begin{align}
	U^0=\langle K_i^\pm : 1\leq i \leq n \rangle,&&  U^+=\langle E_\alpha: \alpha\in\overline{\Phip}\rangle, && \text{and} && U^-=\langle F_\alpha: \alpha\in \overline{\Phip}\rangle.
	\end{align} 
	
	We recall that $<_{br}$ is convex \cite{FeiginVassiliev}, i.e. if $\alpha,\beta,\alpha+\beta\in \Phip$ and $\alpha<_{br}\beta$ then ${\alpha<_{br}\alpha+\beta<_{br}\beta}$. Thus, the following subalgebras are well defined for any $\alpha<_{br}\beta$: 
	\begin{align}
	U^+_{\alpha\beta}=\langle E_\gamma: \alpha<_{br}\gamma<_{br}\beta\rangle && \text{and} &&
	U^-_{\alpha\beta}=\langle F_\gamma: \alpha<_{br}\gamma<_{br}\beta\rangle.
	\end{align} 
	Moreover, 
	\begin{align}
	U^+_{\beta}=\langle E_\gamma: \gamma<_{br}\beta\rangle && \text{and} &&
	U^-_{\beta}=\langle F_\gamma: \gamma<_{br}\beta\rangle
	\end{align} 
	are well defined subalgebras. \newline 
	
		Given an ordering $<_{br}$ on $\overline{\Phip}$, we define a lexicographical ordering on $\Psi=\{0,1\}^{\overline{\Phip}}$. We say that $\psi_1<\psi_2$ if there exists $\alpha\in\overline{\Phip}$ such that $\psi_1(\alpha)=0$, $\psi_2(\alpha)=1$, and $\psi_1(\beta)=\psi_2(\beta)$ for all $\beta>_{br}\alpha$. Observe that $<$ is a total ordering on $\Psi$ with $(1\dots1)$ maximal. 
	
	\begin{defn}
		To each $F^\psi$, we assign the simple root sum \begin{align}
	rt(F^\psi)=\sum_{\alpha\in\overline{\Phip}} \psi(\alpha)h(\alpha).
	\end{align}
	Here, we recall that $h(\alpha)$ is the height of $\alpha\in \overline{\Phip}$, which is given by $\sum c_j^\alpha$ for ${\alpha=\sum c_j^\alpha \alpha_j}$. We remark that $rt(F^\psi)$ can be interpreted as the $l^2$-inner product of the functions $\psi$ and $h$ on $\overline{\Phip}$, and in this sense $rt(F^\psi)=h^*(\psi)$.
	\end{defn}	
	
	Maximality of $(1\dots1)$ implies that $rt(F^{(1\dots1)})\geq rt(F^\psi)$ for all $\psi\in  \Psi$. Since $\Psi$ is totally ordered, the inequality is an equality if and only if $\psi=(1\dots 1)$. Observe that $rt(F^{\psi_1}F^{\psi_2})=rt(F^{\psi_1})+rt(F^{\psi_2})$. By maximality, $rt(F^{\psi_1})+rt(F^{\psi_2})>rt(F^{(1\dots1)})$ implies $F^{\psi_1}F^{\psi_2}=0$. Suppose $rt(F^{\psi})=rt(F^{(1\dots 1)})$ and $\psi(\alpha)=0$ for some $\alpha\in \overline{\Phip}$. Since $(1\dots 1)$ is the unique maximal element of $\Psi$, $F^\psi=0$. \newline 
	
	We also introduce the following notations. For any $\psi\in \Psi$, define $\psi_{<\gamma}\in \Psi$ so that $\psi_{<\gamma}(\alpha)=\psi(\alpha)$ for all $\alpha<\gamma$ and is zero otherwise. The root vectors associated to $\psi_{<\gamma}$ are denoted $F^\psi_{<\gamma}$ and $E^\psi_{<\gamma}$. These maps and root vectors are notated analogously for the other inequalities. 	For each $\alpha\in\overline{\Phip}$ we define $\delta_\alpha\in\Psi$ so that $\delta_\alpha(\alpha)=1$ and is zero otherwise.
	
	\begin{lem}\label{lem:compl}
		Let $\psi\in \Psi$ and $\psi'=1-\psi.$ Then $F^\psi F^{\psi'}$ is a nonzero multiple of $F^{(1\dots 1)}$. Similarly, $E^\psi E^{\psi'}\in \bracks{E^{(1\dots 1)}}$ is nonzero.
	\end{lem}
	\begin{proof}
		Let $k=\sum_{\alpha\in\overline{\Phip}}\psi(\alpha).$ If $k=0$, then $\psi'=(1\dots 1)$ and the claim holds immediately. Suppose that the claim holds for some $k< \text{rank}(\g)$, we show that it is also true for $\sum_{\alpha\in\overline{\Phip}}\psi(\alpha)=k+1$. Fix such a $\psi$. Suppose that $F^\psi=F^\psi_{<\beta} F_{\beta}$. Then 
	\begin{align}
	F^\psi F^{\psi'}&=
	F^\psi_{<\beta} F_{\beta}F^{\psi'}\\
	&=
	F^\psi_{<\beta}\left(a_0F^{\psi'}_{<\beta}F_{\beta}
	+\sum a_{\varphi}F^\varphi_{<\beta}\right)F^{\psi'}_{> \beta}\\
	&=
	a_0F^\psi_{<\beta}F^{\psi'}_{<\beta}F_{\beta}F^{\psi'}_{> \beta}+
	\left(\sum b_{\varphi}F^\varphi_{<\beta}\right) F^{\psi'}_{> \beta}.
	\end{align}
	with $a_{\varphi},b_{\varphi}\in \Q(\zeta)$ and $a_{0}$ equal to a power of $\zeta$. The third equality above is a consequence of the fact $F^\psi_{<\beta}$ and $F^\varphi_{<\beta}$ belong to $U^-_\beta$ for all $\varphi_{<\beta}\in\Psi$, which implies $F^\psi_{<\beta}F^\varphi_{<\beta}\in U^-_\beta$. By induction,
	$F^\psi_{<\beta}F^{\psi'}_{<\beta}F_{\beta}F^{\psi'}_{> \beta}=F^\psi_{<\beta}F^{\psi'+\delta_\beta}$ is a nonzero multiple of $F^{(1\dots 1)}$. Moreover, for each nonzero $b_\varphi$ in the above sum, 
	$rt(F^\varphi_{<\beta} F^{\psi'}_{> \beta})=rt(F^{1\dots1})$ and $\varphi_{<\beta}(\beta)+\psi'_{> \beta}(\beta)=0$. Therefore, $F^\varphi_{<\beta} F^{\psi'}_{> \beta}=0$. Thus, $F^\psi F^{\psi'}$ is a nonzero multiple of $F^{(1\dots 1)}$ for $\sum_{\alpha\in\overline{\Phip}}\psi(\alpha)=k+1\leq \text{rank}(\g)$. The proof for $E^\psi$ is analogous.
	\end{proof}
 For each $\psi\in\Psi$, we define ${\chi_\psi:U^-\to \Q(\zeta)}$ so that $\chi_\psi(F)$ is the coefficient of $F^\psi$ in the PBW basis expression for any $F\in U^-$.
\begin{lem}\label{lem:triangle}
	Fix $\psi_1,\psi_2\in\Psi$ such that $\psi_1<\psi_2$. Then $\chi_{(1\dots 1)}(F^{1-\psi_2}F^{\psi_1})=0$ and $\chi_{(1\dots 1)}(E^{1-\psi_2}E^{\psi_1})=0$.
\end{lem}
\begin{proof}
	We prove the result in $U^-$, the proof is similar in $U^+$. We assume that $\psi_1$ and $\psi_2$ are chosen so that $rt(F^{1-\psi_2})+rt(F^{\psi_1})=rt(F^{(1\dots 1)})$.
	Let $\gamma$ be the root satisfying the properties $\psi_1(\gamma)=0$, $\psi_2(\gamma)=1,$ and  $\psi_1(\beta)=\psi_2(\beta)$ for all $\beta>_{br}\gamma$. Then,
	$F^{1-\psi_2}F^{\psi_1}=(F^{1-\psi_2}_{<\gamma}F^{1-\psi_1}_{>\gamma})(F^{\psi_1}_{<\gamma}F^{\psi_1}_{>\gamma}).$ Since $<_{br}$ is a convex ordering, we have the following equalities: 
	\begin{align}
	F^{1-\psi_1}_{>\gamma}F^{\psi_1}_{<\gamma}&=
	a_0F^{\psi_1}_{<\gamma}F^{1-\psi_1}_{>\gamma}+\sum_{\substack{\psi_1(\alpha)=1\\\alpha<\gamma}}a_\alpha F^{\psi_1}_{<\alpha}F^{1-\psi_1}_{>\gamma}F_\alpha F^{\psi_1}_{>\alpha,<\gamma}\\&=
	a_0F^{\psi_1}_{<\gamma}F^{1-\psi_1}_{>\gamma}+\sum_{\varphi_{>\gamma}<(1-\psi_1)_{>\gamma}}b_\varphi F^{\psi_1}_{<\alpha}F^{\varphi} F^{\psi_1}_{>\alpha,<\gamma}=
	\sum_{\varphi_{\geq\gamma}<(1-\psi_1)_{\geq\gamma}}c_\varphi F^{\varphi}
	\end{align} 
	for some $a_0,a_\alpha,b_\varphi,c_\varphi\in \Q(\zeta)$. We note that
	\begin{align}
	\{\varphi\in\Psi:\varphi_{>\gamma}<(1-\psi_1)_{>\gamma} \}=
	\{\varphi\in\Psi:\varphi<(1-\psi_1)_{>\gamma} \}=
	\{\varphi\in\Psi:\varphi_{>\gamma}<(1-\psi_1) \}.
	\end{align}
	Therefore,
	\begin{align}
	F^{1-\psi_2}F^{\psi_1}=
	F^{1-\psi_2}_{<\gamma}
	\left(	\sum_{\varphi_{\geq\gamma}<(1-\psi_1)_{\geq\gamma}}c_\varphi F^{\varphi}\right)
	F^{\psi_1}_{>\gamma}=\sum_{\varphi_{\geq\gamma}<1_{\geq\gamma}} a_{\varphi}F^{\varphi}
	\end{align}
	for some $a_{\varphi}\in\Q(\zeta)$. In each summand $F^{\varphi}$ above, there is a root $\beta\geq\gamma$ such that $\varphi(\beta)=0$. Since $rt(F^{1-\psi_2})+rt(F^{\psi_1})=rt(F^{(1\dots 1)})$,  $F^{1-\psi_2}F^{\psi_1}=0$. In particular, ${\chi_{(1\dots 1)}(F^{1-\psi_2}F^{\psi_1})=0}$.
\end{proof}
\begin{cor}\label{cor:lowest}
	For any nonzero vector $F\in U^-$, there exists ${\psi}\in\Psi$ such that $F^{\psi}F$ is a nonzero multiple of $F^{{(111)}}$. Similarly for $U^+$.
\end{cor}
\begin{proof}
	Fix $F\in U^-$.	We express $F$ using the PBW basis,
	\[F=\sum_{\psi\in\Psi} a_\psi F^\psi. \]
	Consider the set $A=\{\psi\in\Psi:a_\psi\neq 0\}$ and let $x=\min\limits_{\psi\in A}rt(\psi)$. Fix $\psi^*\in A$ to be maximal with respect to $<$ such that $rt(\psi^*)=x$.
	By Lemma \ref{lem:compl}, $F^{1-\psi^*}F^{\psi^*}$ is a nonzero multiple of $F^{{(111)}}$.
	Fix $\psi\in A$. If $rt(\psi)>x$, then $rt(F^{1-\psi^*}F^\psi)>rt(F^{(1\dots1)})$, which implies $F^{1-\psi^*}F^\psi=0$. If $rt(\psi)=x$, we may apply Lemma \ref{lem:triangle},  which tells us $F^{1-\psi^*}F^\psi=0$. Thus,
	\[
	F^{1-\psi^*}F
	=F^{1-\psi^*}\left(\sum_{\psi\in A} a_\psi F^\psi\right)
	=a_{\psi^*} F^{1-\psi^*}F^{{\psi^*}}\in\bracks{F^{(111)}}
	\]
	is nonzero.
\end{proof}
\begin{lem}\label{lem:EFtype}
	For each $\alpha<\beta\in{\overline{\Phip}}$, $E_\alpha F^{(1\dots 1)}_{\geq\beta}=
	\sum c_{\psi,\varphi}F^{\psi}_{\geq\beta}E^{\varphi}_{\leq\alpha}$ with each $\varphi_{\leq\alpha}\neq 0$. 
\end{lem}
\begin{proof}
	We induct downwards on the index of $\beta\in {\overline{\Phip}}$ with respect to $<_{br}$. Suppose that $\beta$ is maximal in $\overline{\Phip}$, then $E_\alpha F_{\geq\beta}^{(1\dots 1)}=E_\alpha F_\beta$. In the PBW basis, this is a linear combination of $F_\beta E_\alpha$ and $E^\varphi$ with $\varphi<\delta_\gamma$ for various $\gamma<\alpha$. Thus, the claim holds in this case. Suppose the result holds for $\beta$ of index greater than $k>1$, we prove that it also holds for $\beta$ of index $k-1$ and any $\alpha<\beta$.  We have
\[	
	E_\alpha F_{\geq\beta}^{(1\dots 1)}=E_\alpha F_\beta F_{>\beta}^{(1\dots 1)}
	=\left(F_\beta E_\alpha+\sum_{\substack{\alpha< \gamma<\beta\\\varphi<\delta_\gamma}} c_{\psi\varphi}F^\psi E^\varphi+\sum_{0<\psi_{>\beta }} c_\psi F^\psi \right)F_{>\beta}^{(1\dots 1)}.\]
By induction, the first two terms above can be expressed in the desired form. The expression $\sum_{0<\psi_{>\beta }} c_\psi F^\psi$ can arise when $h(\alpha)<h(\beta)$. By maximality of $(1\dots1)$, $F_\gamma F_{>\beta}^{(1\dots 1)}=0$ for each $\gamma>\beta$. Therefore, the last term vanishes and the remaining sum has the desired form.
\end{proof}
\section{Induced Representations}\label{sec:Vt}
In this section, we define the representations $\Vt$ as being induced by $\Uqg$ with respect to a character on the Borel subalgebra. We then prove various results for $q=\zeta$. In Proposition \ref{prop:indecomp}, we show that $\Vt$ is indecomposable for all $\t\in\P$. We end this section by giving a characterization of reducibility of $\Vt$.  We specialize our discussion of $\Vt$ to the $\slthree$ case in Section \ref{sec:props}. Throughout this section, we assume $\g$ is a semisimple Lie algebra and following Definition \ref{defn:Vt}, $\zeta$ is a primitive fourth root of unity. \newline 

We denote the Borel subalgebra by
	$B$, which is the subalgebra generated by $U^0$ and $U^-$. Using the PBW basis, we have $\Uzg\cong U^-\otimes B$. \newline
	
	We now define the representation $\Vt$ as a Verma module over $\Uqg$ at a primitive $l$-th root of unity. Note that the group of characters $\P$ on $U^0$ is isomorphic to $(\Cx)^{n}$. Each character $\t=(t_1,\dots, t_n)$ is determined by the images $t_i$ of $K_i$ in  $\Cx$. The character $\t$ extends to a character $\gamma_{\t}:B\rightarrow \C$ by \begin{align}\label{eq:char}
	\gamma_{\t}(K_i)=t_i,  &&\gamma_{\t}(E_i)=0.
	\end{align}
	
	\begin{defn}\label{defn:Vt}
		Let $\gamma_{\t}:B\rightarrow\C$ be a character as in {(\ref{eq:char})}. 
		Let $V_{\t}=\langle v_h\rangle$ be the \hbox{1-dimensional} left $B$-module determined by $\gamma_{\t}$, i.e. for $b\in B$, $bv_h=\gamma_{\t}(b)v_h$.
		We define the representation $V(\t)$ to be the induced module 
		\begin{align}
		\Vt={\text{{Ind}} _{B}^{\Urg}(V_{\t})=\Urg\otimes_{B}V_{\t}}. 
		\end{align}
	\end{defn}
	
	In the $\slthree$ case, we will consider various other induced representations coming from characters of this type. See Appendix \ref{Ind} for more details on the induction functor and induced modules. 
\newline 

		Note that $\Vt$ and $U^-$ are isomorphic as vector spaces. The PBW basis of $U^-$ extends to a basis
		of $\Vt$ by tensoring $v_h$ and $U^-$ acts on these basis vectors accordingly. In particular, this action is independent of $\t$. We denote the lowest weight vector $F^{(1\dots 1)}v_h$ by $v_l$. \newline 
		
		For the remainder of this section, we assume $q=\zeta$. Then $\Vt$ has dimension $|\Psi|=2^{|\overline{\Phip}|}$ with basis determined by \hbox{Corollary \ref{cor:basis}.} In types $ADE$, $\overline{\Phip}=\Phip$ and  so $\Vt$ has dimension $2^{|{\Phip}|}$.

		\begin{prop}\label{prop:indecomp}
		For all $\t\in\P$, $\Vt$ is an indecomposable representation. 
	\end{prop}
	\begin{proof}
		Suppose that $\Vt$ admits a direct sum decomposition $W_1\oplus W_2$ with respect to the $\overline{U}$ action. Fix non-zero vectors $w_1\in W_1$ and $w_2\in W_2$. By Lemma \ref{cor:lowest}, there exists ${\psi}\in \Psi$ so that $F^{{\psi}}w_1$ is a nonzero multiple of $v_l$. Thus, $v_l\in W_1$. The same argument applies to $w_2$, and so $v_l\in W_2$. Hence, $\langle v_l\rangle$ is a subspace in $W_1\cap W_2$, which contradicts the existence of a direct sum decomposition. Thus, $\Vt$ is indecomposable.
	\end{proof}
	\begin{rem}
		For every pair of distinct characters $\t,\s\in\P$, $\Vt\not\cong \Vs$. This is clear since the highest weight determines the representation. Thus, the representations $\Vt$ form an infinite family of representation classes.
	\end{rem}
	The actions of $K_i$ break $\Vt$ into weight spaces. Recall that multiplication in $\mathcal{P}$ is entrywise. To each $\psi\in\Psi$ we assign $\mvec{\sigma^\psi}\in\mathcal{P}$ defined so that 
$K_iF^\psi v_h=\mvec{\sigma^\psi t}(K_i) F^\psi v_h$. 
More precisely,  \begin{align}\label{eq:xipowersforsigma}
\mvec{\sigma^\psi}(K_i)=\zeta^{(\alpha_i, \sum \psi(\alpha)\alpha)}.
\end{align}  
We define 
\begin{align}\label{defn:sigma}
\Sigma=\{\mvec{\sigma^\psi}:\psi\in\Psi \}\subseteq\P
\end{align} to be the weights of $V(\mvec{1})$.
Therefore, the  weight spaces of $\Vt$ are labeled by $\Sigma\t$. Note that the map $\Psi\to\Sigma$ given by $\psi\mapsto \mvec{\sigma^\psi}$ is not an injection in general. \newline 

Using the weight space data above, we describe the remaining actions of the induced $\U$-module.
The action of $U^+$ is defined to be zero on $v_h$, but the commutation relations $[E_i,F_i]=\floor{K_i}$ imply that it does not act trivially on all of $\Vt$. For specific values of $\t$, we will see that some matrix entries of $E_i$ vanish.

\begin{defn}
	The \emph{irreducibility vector} of a representation $\Vt$ is the vector \[\Omega=E^{{(1\dots 1)}}v_l=E^{{(1\dots 1)}}F^{{(1\dots 1)}}v_h.\]
	
\end{defn}
We will prove in Corollary \ref{cor:irred} that $\Vt$ is reducible if and only if $\Omega$ vanishes.
\begin{rem}
	If $\Omega$ is non-zero, then it is a highest weight vector. 
\end{rem}

\begin{prop}\label{prop:computation}
	For each semisimple Lie algebra $\g$, we have the equality \begin{align}
	\Omega= \prod_{\alpha\in \overline{\Phip}}\left(\zeta^{\sum_{\beta>_{br}\alpha}(\alpha,\beta)}\floor{\zeta^{-\sum_{\beta>_{br}\alpha}(\alpha,\beta)}K_\alpha}\right)v_h.
	\end{align}
\end{prop}
\begin{proof}
	We compute $\Omega$ directly. Observe that by maximality of $(1\dots 1)$, for each $\alpha\in \overline{\Phip}$, $E_\alpha E^{(1\dots1)}_{>\alpha}=\left(\zeta^{\sum_{\beta>_{br}\alpha}(\alpha,\beta)}\right) E^{(1\dots1)}_{>\alpha}E_\alpha$. 
	By Lemmas \ref{lem:Kcomm} and \ref{lem:EFtype},
	\[E_\alpha F^{(1\dots 1)}_{\geq \alpha}v_h=E_\alpha F_\alpha F^{(1\dots 1)}_{> \alpha}v_h=
	\left(F_\alpha E_\alpha +\floor{K_\alpha }\right)F^{(1\dots 1)}_{>\alpha}v_h
	=F^{(1\dots 1)}_{>\alpha}\floor{\prod_{\beta>\alpha}\zeta^{-(\alpha,\beta)}K_\alpha}v_h. \]
	Therefore,
	\[
	\Omega=
	\left(\prod_{\alpha\in \overline{\Phip}} E_\alpha \right)
	\left(\prod_{\alpha\in \overline{\Phip}} F_\alpha \right)v_h=
	\left(\prod_{\alpha\in \overline{\Phip}}^{reverse} \left(\zeta^{\sum_{\beta>_{br}\alpha}(\alpha,\beta)}\right)E_\alpha\right)
	\left(\prod_{\alpha\in \overline{\Phip}} F_\alpha \right)v_h\]\[
	=\prod_{\alpha\in \overline{\Phip}}\left(\zeta^{\sum_{\beta>_{br}\alpha}(\alpha,\beta)}\floor{\prod_{\beta>\alpha}\zeta^{-(\alpha,\beta)}K_\alpha}\right)v_h.
	\]
\end{proof}

Let $\X_\alpha=\left\lbrace\t\in\P:\floor{\zeta^{-\sum_{\beta>_{br}\alpha}(\alpha,\beta)}\t(K_\alpha)}=0\right\rbrace$, which determines a variety in $\P$. We then define the algebraic set $\R=\bigcup_{\alpha\in\overline{\Phip}} \X_\alpha$.
\begin{lem}\label{lem:det}
	The collection of vectors
	\[\mathscr{E}= \{E^\psi v_l:\psi\in\Psi\}\] 
	forms a basis for $\Vt$ if and only if $\Omega\neq0$.
\end{lem}
\begin{proof}
	The last vector in $\mathscr{E}$, with respect to $<$, is $\Omega$. Thus, $\Omega=0$ implies $\mathscr{E}$ is not a basis. \newline
	
	Suppose now that $\mathscr{E}$ does not form a basis. Then there exists a nonzero element $E\in U^+$ that yields a linear dependence $Ev_l=0.$
	By Lemma \ref{cor:lowest}, there exists  ${\psi}\in \Psi$ such that $E^{{\psi}}E=cE^{{(111)}}$ for some nonzero $c\in\Q_4$. Then
	\begin{align*}
	&\Omega=E^{(1\dots 1)}v_l=cE^{{\psi}}Ev_l=0.\qedhere
	\end{align*} 
\end{proof}
\begin{cor}\label{cor:irred}
	The following are equivalent:
	\begin{center}
		\begin{tabular}{ccccc}
		$\bullet$~$\Vt$ is irreducible &\hspace*{5em}& $\bullet$~$\Omega\neq0$ &\hspace*{5em}& $\bullet$~$\t\notin\R$.
	\end{tabular}
	\end{center}
\end{cor}
\begin{proof}
	Irreducibility holds if and only if any nonzero $v\in\Vt$ is a cyclic vector for the module. That is to say, the action of $\U$ on $v$ generates $\Vt$. Fix any $v\neq0$. By \hbox{Lemma \ref{cor:lowest}}, we may assume $v=v_l$. Raising this lowest weight vector $v_l$ by each $E^\psi$, we obtain the vectors of $\mathscr{E}$, which we claim to be a basis of $\Vt$. Equivalently, by Lemma \ref{lem:det}, we verify that $\Omega$ is nonzero. From Proposition \ref{prop:computation}, $\Omega=0$ {exactly} when $\floor{\prod_{\beta>\alpha}\zeta^{-(\alpha,\beta)}K_\alpha}v_h=0$ for some $\alpha\in \overline{\Phip}$. It follows that $\U$ acting on $v$ generates $\Vt$ if and only if $\t\notin\R$. Since this holds for every non-zero $v\in\Vt$, we have proven the claim. 
\end{proof} 
\begin{rem}
	By Lemma \ref{lem:central}, the $\bracks{K_i^4:i\in\{1,\dots,n\}}$ is a subalgebra of central Cartan elements. Let $\mathcal{C}$ be the category of finite dimensional representations on which each $K_i$ acts diagonally. Let $\mvec{a}\in\P$,  and let $\mathcal{C}_{\mvec{a}}\subseteq \mathcal{C}$ be the subcategory on which $K_i^4=\mvec{a}(K_i)1$. Then  
	$
	\mathcal{C}=\bigoplus_{\mvec{a}\in\P}\mathcal{C}_{\mvec{a}}.
	$
	Since each $K_i$ is group-like, we have 
	$
	\mathcal{C}_{\mvec{a}}\otimes \mathcal{C}_{\mvec{b}}\subseteq \mathcal{C}_{\mvec{ab}}$ 
	for every $\mvec{a},\mvec{b}\in\P$.
	The category $\mathcal{C}_{\mvec{a}}$ contains the representations $\{\Vt:\t^4=\mvec{a} \}$. In particular,  $\mathcal{C}_{\mvec{1}}$ is non-semisimple. For each $\mvec{a}\in\P$,
	$
	\mathcal{C}_{\mvec{1}}\otimes \mathcal{C}_{\mvec{a}}\subseteq \mathcal{C}_{\mvec{a}},$
	and so each $\mathcal{C}_{\mvec{a}}$ is non-semisimple.
\end{rem}

	\section{Semisimple Tensor Products}\label{sec:directsum}
	Here, we introduce the notion of non-degeneracy to characterize complete reducibility of the representations $\Vt\otimes\Vs$. Given such a tensor product, we prove it is isomorphic to a direct sum of irreducible representations provided that $V(\mvec{\sigma ts})$ is irreducible for all $\mvec{\sigma}\in\Sigma$. Recall that $\Sigma$ consists of the weights of $V(\mvec{1})$, see \hbox{(\ref{defn:sigma})}. The goal of this section is to prove Theorem \ref{thm:directsum}. Throughout this section, we assume $q$ is a primitive fourth root of unity $\zeta$ and that $\g$ is a semisimple Lie algebra, unless stated otherwise.
	
	\begin{defn} A pair $(\t,\s)\in\P^2$ is called \emph{non-degenerate} if $\mvec{\sigma ts}$ is irreducible for all $\mvec{\sigma}\in \Sigma$. We call $\Vt\otimes\Vs$ a \emph{non-degenerate representation} if 
		$(\t,\s)$ non-degenerate.
	\end{defn}
	
	\directsum*
	
	Our proof of the theorem relies on finding highest weight vectors in the tensor product and looking at their image under $U^-$. Generically, each of these cyclic subspaces is isomorphic to an induced representation and is identified by the weight of its highest weight vector. These vectors can be more easily described in $\text{{Ind}} _{B}^{\U}\left(V_{\t}\otimes \text{{Ind}} _{B}^{\U}(V_{\s})\right)$, which by Proposition \ref{prop:iso}, is isomorphic to $\Vt\otimes\Vs$.  Denote $\Vts=\text{{Ind}} _{B}^{\U}\left(V_{\t}\otimes \text{{Ind}}_{B}^{\U}(V_{\s})\right)$.
	\newline
	
	Recall from the construction of $\Vt$ that $\gamma_{\t}$ is the character which determines the action of $B$ on $v_h$. For $a\in\U$, let $\Delta(a)=a'\otimes a''$ be the coproduct of $a$ with the implicit summation notation. By definition, 
	\begin{align}\label{eq:identify}
	V(\t) \hotimes  V(\s)&=\U\otimes_{B}\left(V_{\t}\otimes (\U\otimes_{B} V_{\s})\right)\cong \left(\U\otimes \left(V_{\t}\otimes (\U\otimes V_{\s})\right)\right)/Q\cong \left(\U\otimes \U\right)/Q'
	\end{align}
	with
	\begin{align}
	Q&=\langle a_1b_1\otimes \left(v_h\otimes (a_2b_2\otimes v_h)\right)-a_1\otimes \left(b_1'.v_h\otimes (b_1''a_2\otimes b_2.v_h)\right):a_i\in\U, b_i\in B\rangle\\
	&\cong \langle a_1(b_1v_h\otimes a_2b_2v_h) -\gamma_{\t}(b_1')\gamma_{\s}(b_2)a_1v_h\otimes b_1''a_2v_h:a_i\in \U, b_i\in B\rangle\\
	&=Q'
	\end{align} 
	and the above isomorphisms suppress tensoring of 1-dimensional vector spaces $V_{\t}$ and $V_{\s}$. We include $v_h$ in the notation for vectors in $\Vts$ to avoid confusion with the algebra $\U\otimes\U$ i.e. a vector $v=a_1\otimes(v_h\otimes (a_2\otimes v_h))\in\Vts$ will be denoted by $a_1(v_h\hotimes a_2v_h)$ under the identification in (\ref{eq:identify}). The action of $\U$ is by left multiplication on the first tensor factor, which may then be simplified. An example of the action in the $\slthree$ case is provided below. \begin{ex}\label{ex}
		The action of $F_1E_1$ on $v_h\hotimes F^{(101)} v_h$ is given as follows:
		\begin{align}
		F_1E_1.(v_h\hotimes  F^{(101)} v_h)
		&=F_1(\gamma_{\t}(E_1)v_h\hotimes K_1F^{(101)} v_h+v_h\hotimes E_1F^{(101)} v_h)\\
		&=F_1v_h\hotimes \floor{\zeta s_1}F^{(001)}v_h.
		\end{align}
	\end{ex}
	We fix a basis on $\Vts$,
	\begin{align}\label{eq:basis2}
	\{F^\psi v_h\hotimes F^{\psi'}v_h:\psi,\psi'\in\Psi\},
	\end{align} 
	which is the standard tensor product basis given by the PBW basis of $U^-$ in each factor. We see that the tensor product representation has dimension $|\Psi|^2.$
	\begin{lem}\label{lem:inclusions}
		Let $(\t,\s)\in\P$ be a non-degenerate pair and $\mvec{\sigma^\psi}\in \Sigma$. Then the subspace
		\begin{align}
		\overline{V}_{\mvec{\sigma^\psi}}:=\langle F^\varphi( \Omega\hotimes F^\psi v_h):\varphi\in\Psi \rangle \subseteq \Vts
		\end{align}
		and $V(\mvec{\sigma^\psi}\t\s)$ are isomorphic as $\U$-modules.
	\end{lem}
	\begin{proof}
		Following Proposition \ref{prop:computation},
		\begin{align*}
			\Omega\hotimes  F^\psi v_h
			=\prod_{\alpha\in \overline{\Phip}}\left(\zeta^{\sum_{\beta>_{br}\alpha}(\alpha,\beta)}\floor{\zeta^{-\sum_{\beta>_{br}\alpha}(\alpha,\beta)}\mvec{\sigma^\psi ts}(K_\alpha)}\right)v_h\hotimes F^\psi v_h+\sum_{\psi'\neq 0} c_{\psi'\psi''} F^{\psi'} v_h\hotimes F^{\psi''} v_h
		\end{align*}
		
		for some $c_{\psi'\psi''}\in \Q(t_i,s_i,\zeta)$. Having assumed non-degeneracy, the $v_h\hotimes  F^{\psi} v_h$ component of $\Omega\hotimes  F^{\psi} v_h$ is non-zero. Therefore, $ \Omega\hotimes  F^{\psi} v_h$ is a highest weight vector of weight $\mvec{\sigma^\psi}\t\s$ and $\overline{V}_{\mvec{{\sigma}^\psi}}$ is an irreducible $|\Psi|$-dimensional subrepresentation of $\Vts$. Thus, by the irreducibility of $V(\mvec{\sigma^\psi }\t\s)$, the map which sends $v_h\in V(\mvec{\sigma^\psi}\t\s)$ to $\Omega\hotimes  F^{\psi} v_h\in \overline{V}_{\mvec{\sigma^\psi}}$ determines an isomorphism.
	\end{proof}
	\begin{lem}\label{lem:iso}
		Let $(\t,\s)$ be a non-degenerate pair. Then $\Vts$ is isomorphic to the direct sum $\bigoplus_{\psi \in\Psi} \overline{V}_{\mvec{\sigma^\psi}}$.
	\end{lem}
	\begin{proof}
		Observe that the non-degeneracy assumption on $(\t,\s)$ implies the irreducibility of each $\overline{V}_{\mvec{\sigma^\psi}}$. Every   $\overline{V}_{\mvec{\sigma}^\psi}$ includes into $\Vts$ as the subspace generated by $U^-$ acting on the highest weight vector $ \Omega\hotimes  F^{\psi} v_h$. Since each $ \Omega\hotimes  F^{\psi} v_h$ is distinct and each $V_{\mvec{\sigma^\psi}}$ is irreducible, $V_{\mvec{\sigma^\psi}}\cap V_{\mvec{\sigma^{\psi'}}}=\langle0\rangle$ for $\psi\neq\psi'$.  Hence, $\bigoplus_{\psi \in\Psi} \overline{V}_{\mvec{\sigma^\psi}}$ injects into $\Vts$. By dimensionality, this injection is a surjection and, therefore, an isomorphism.
	\end{proof}
	Let $\widetilde{\Gamma}$ denote the isomorphism $\bigoplus_{\psi\in\Psi} V_{\mvec{\sigma^\psi}}\cong \Vts$ described in Lemma \ref{lem:iso}. Using the aforementioned lemmas we prove the first main theorem.
	\begin{proof}[Proof of Theorem \ref{thm:directsum}]
		We construct an intertwiner $\Gamma$ in the following diagram when $(\t,\s)$ is a non-degenerate tuple.
		\[
		\begin{tikzcd}[]
		\Vt\otimes\Vs
		\arrow{r}{\Theta}\arrow[swap]{dr}{\Gamma} &\arrow{d}{\widetilde{\Gamma}}
		\Vts\\    
		& \bigoplus_{\mvec{\sigma^\psi}\in\Psi}V({\mvec{\sigma^\psi}}\t\s)
		\end{tikzcd}
		\]
		We see that $\Gamma=\widetilde{\Gamma}\circ\Theta$ is given by a composition of isomorphisms. The above lemmas establish that  $\widetilde{\Gamma}$ is an isomorphism for non-degenerate tuples. Moreover, $\Theta$ is an isomorphism by Proposition \ref{prop:iso}, which is independent of $\t$ and $\s$. This proves the theorem.
	\end{proof}
	\begin{rem}
		We see that a tensor product of indecomposable, but reducible, representations may decompose into a direct sum of irreducibles. For example, when $\g=\slthree$, $V(t_1,1)\otimes V(1,s_2)$ is a non-degenerate tensor product representation for generic $t_1$ and $s_2$.
	\end{rem}
	Observe that the isomorphism class of a non-degenerate tensor product depends only on the product $\t\s$. We define an action of $\P$ on $\P^2$ which preserves the product $\mvec{ts}$ as follows. Let $\mvec{\lambda},\t,\s\in \P$ and set
	\begin{equation}\label{eq:action}
	\mvec{\lambda}\cdot(\mvec{t},\mvec{s})=(\mvec{\lambda t},\mvec{\lambda^{-1}s}).
	\end{equation}

	\begin{cor}\label{cor:transfer}
		Let $(\t,\s)$ be a non-degenerate tuple. For any $\mvec{\lambda}\in \mathcal{P}$ such that $\mvec{\lambda}\cdot(\t,\s)$ is also non-degenerate, then
		\begin{equation}\label{eq:transfer}
		\Vt\otimes\Vs\cong V(\mvec{\lambda t})\otimes V(\mvec{\lambda^{-1}s}).
		\end{equation}
	\end{cor}
We examine when this isomorphism holds more generally in the $\slthree$ case in Sections \ref{sec:cyclicgen}, \ref{sec:cyclic1}, \ref{sec:cyclici}, and \ref{sec:thms}. 
\section{Specialization to $\g=\slthree$}\label{sec:props}
In this section, we describe the structure of $\Uzslthree$ explicitly and specialize the results of Sections \ref{sec:notation} and \ref{sec:Vt} to $\g=\slthree$. In particular, we give an explicit description of $U^+$ on $\Vt$ in Table \ref{table:actions} and give an explicit description of the algebraic set $\R$ on which $\Vt$ is reducible. \newline 

\comm{Recall our convention for ordering the roots of $\overline{\Phip}$:
\begin{align}
\alpha_1<_{br}\alpha_1+\alpha_2<_{br}\alpha_2,
\end{align}
which also yields a lexicographical ordering on $\Psi=\{0,1\}^{\overline{\Phip}}$. } We will use $\alpha_{12}$ to denote the root $\alpha_1+\alpha_2$. We define
\begin{align}
F^\psi=F^{(\psi_{(1)},\psi_{(12)},\psi_{(2)})}=F_1^{\psi(\alpha_1)}F_{12}^{\psi(\alpha_{12})}F_2^{\psi(\alpha_2)}.
\end{align}  
\comm{Recall that 	
\begin{align}
E_{12}=-(E_1E_2+\zeta E_2E_1)  &&\text{and}&& F_{12}=-(F_2F_1-\zeta F_1F_2).
\end{align}}
The relations $F_{12}^2=(F_2F_1-\zeta F_1F_2)^2=0$ and $E_{12}^2=(E_1E_2+\zeta E_2E_1)^2=0$ imply $(F_1F_2)^2=(F_2F_1)^2$ and $(E_1E_2)^2=(E_2E_1)^2$. By Corollary \ref{cor:basis},
\begin{align}\label{eq:basis}
\B&=\{1, F_1, F_2,F_1F_2, F_{12}, F_1F_{12}, F_{12}F_2, F_1F_{12}F_2\}\\&=
\{F^{(000)},F^{(100)},F^{(010)},F^{(110)},F^{(001)},F^{(101)},F^{(011)},F^{(111)}\}
\end{align}
is a basis of $U^-$, which is equipped with the lexicographical ordering $<$. 
The space $\Psi$ is presented in Figure \ref{fig:cube} below. Adding $\delta_1$, $\delta_2$, and $\delta_{12}$, the indicators of a root $\alpha$, is indicated by green, blue, and turquoise segments, respectively.

\begin{figure}[h]\centering
		\begin{subfigure}[b]{0.4\textwidth}
		\centering
		\includegraphics[]{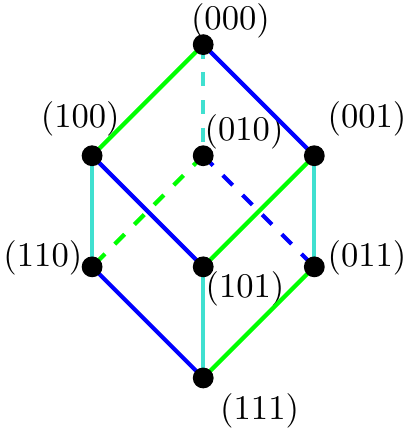}\caption{Visualization of $\Psi$ as a cube.}\label{fig:cube}
	\end{subfigure}
	\quad
\begin{subfigure}[b]{0.4\textwidth}
	\centering
	\includegraphics[]{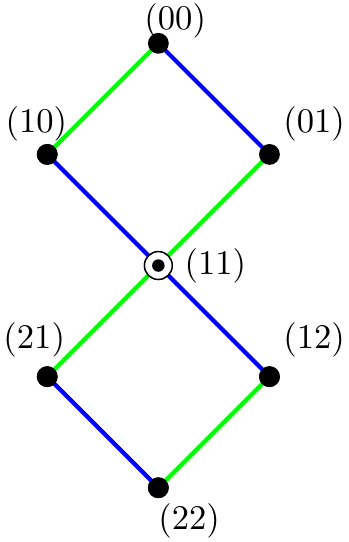}
	\caption{\label{fig:proj} The image of $\Psi$ under $P$.}
\end{subfigure}
\caption{Presentations of $\Psi$.}
\end{figure}

Let $\overline{\alpha}$ be the simple root components of a root $\alpha\in\overline{\Phip}$. In particular, $\overline{\alpha_{12}}=\overline{\alpha_1}+\overline{\alpha_2}$.
This map induces
\begin{align}\label{eq:proj}
P:\Psi&\rightarrow \Z^{\Deltap}
\\(\psi_{(1)},\psi_{(12)},\psi_{(2)})&\mapsto (\psi_{(1)}+\psi_{(12)},\psi_{(2)}+\psi_{(12)})\notag
\end{align}
which can be seen as a projection of the cube in Figure \ref{fig:cube} into the plane and the deletion of turquoise segments. This also determines the weight spaces on representations. \newline

Recall that $\Vt$ is an induced representation and is isomorphic to $U^-$ as vector spaces. We use $\B$ to determine a basis of $\Vt$ by tensoring $v_h$. All vector expressions in $\Vt$ will be expressed using the basis $\B$.\newline

As noted above, the action of $U^+$ on $\Vt$ depends heavily on $\t$. The actions of $E_1$ and $E_2$ are given in terms of $t_1$ and $t_2$ in Table \ref{table:actions} below. Note that when either $t_1^2=1$ or $t_2^2=1$, terms vanish from the expressions in Table \ref{table:actions}. These vanishings are related to the reducibility of the representation. 

\begin{table}[h!]	\caption{Generic nonzero actions of $E_1$ and $E_2$ on $\Vt$ expressed in the induced PBW basis. }
	\begin{align*}
	&E_1F^{(100)}v_h=\floor{t_1}F^{(000)}v_h &&E_2F^{(001)}v_h=\floor{t_2}F^{(000)}v_h\\
	&E_1F^{(101)}v_h=\floor{\zeta t_1}F^{(001)}v_h && E_2F^{(101)}v_h=\floor{t_2}F^{(100)}v_h\\
	&E_1F^{(010)}v_h=\zeta t_1F^{(001)}v_h && E_2F^{(010)}v_h=-t_2^{-1}F^{(100)}v_h\\
	&E_1F^{(110)}v_h=\zeta t_1F^{(101)}v_h-\floor{\zeta t_1}F^{(010)}v_h 
	&&E_2F^{(011)}v_h= t_2^{-1}F^{(101)}v_h+ \floor{t_2}F^{(010)}v_h\\
	&E_1F^{(111)}v_h=\floor{t_1}F^{(011)}v_h &&E_2F^{(111)}v_h= \floor{t_2}F^{(110)}v_h
	\end{align*}
	\label{table:actions}
\end{table}

\comm{Recall that 
\begin{align}
\Omega= \prod_{\alpha\in \overline{\Phip}}\left(\zeta^{\sum_{\beta>_{br}\alpha}(\alpha,\beta)}\floor{\zeta^{-\sum_{\beta>_{br}\alpha}(\alpha,\beta)}K_\alpha}\right)v_h
\end{align}
and if $\Omega=0$, then $\Vt$ is reducible.} In the $\slthree$ case, Proposition \ref{prop:computation} shows that
$
\Omega=	-\zeta\floor{t_1}\floor{t_2}\floor{\zeta t_1t_2}v_h$ and vanishes on the following subsets of $\P$:
\begin{align}
\X_1&=\{\t\in\P :t_1^2=1\},  & 
\X_2&=\{\t\in\P :t_2^2=1\}, &
\H&=\{\t\in\P :(t_1t_2)^{2}=-1\}.
\end{align}
 For $\t\in\H$, the subspace $\bracks{E_1E_2v_l,E_2E_1v_l}$ is 1-dimensional. In particular, $E^{(110)}v_l=0$ and $E^{(011)}v_l=0$. 	Let $\R$ be the union $\X_1\cup\X_2\cup\H$. We partition $\R$ into disjoint subsets indexed by nonempty subsets $I\subsetneq \overline{\Phip}$, with
 \begin{align}
 \R_{I}=\left(\bigcap_{\alpha\in I} \X_\alpha\right)\setminus \left(\bigcup_{\alpha\notin I}\X_\alpha\right).
 \end{align} We define $\R_\emptyset$ to be $\P\setminus\R$, so that  $\{\R_I\}_{I\subsetneq\overline{\Phip}}$ yields a partition of $\P$.\newline

We give a graphical description of $\Vt$ in terms of basis vectors and maps between them by presenting the action of this module on weight spaces, as seen in Figure \ref{fig:action}. Each solid vertex indicates a one dimensional weight space of $\Vt$, and the ``dotted'' vertex indicates the two dimensional weight space spanned by $F^{(101)}v_h$ and $F^{(010)}v_h$. Green edges indicate actions of $E_1$ and $F_1$, and blue edges indicate actions of $E_2$ and $F_2$. All downward arrows represent action of $F_i$, and upward arrows represent the action of $E_i$. For non-generic choices of the parameter $\t$, upward arrows are deleted from the graph because matrix elements of $E_1$ and $E_2$ vanish.

\begin{figure}[h!]
	\centering
	\includegraphics[scale=1.25]{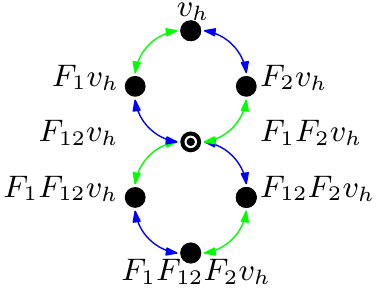}
	\caption{
		The action of $\U$ on the weight spaces of $V(\mvec{t})$ for generic $\t$. }\label{fig:action}
\end{figure}

	Depending on how irreducibility fails, different upward pointing arrows vanish from \hbox{Figure \ref{fig:action}}. The basic cases can be seen in Figure \ref{fig:subs}. In this figure, the representations are color-coded to indicate the quotient representation (gray) and the subrepresentation (red). Actions of $F_i$ which vanish under the quotient are indicated by dotted arrows. Tensor product decompositions involving the 4-dimensional irreducible representations are given in \cite{Harpersl3}.	The correspondence between a choice of $\t$ and the subrepresentation generated by $v_l$ is shown in Table \ref{table:corresp} and Figure \ref{fig:corresp}.

\begin{figure}[h!]
	\centering
	\begin{subfigure}[b]{0.3\textwidth}
		\centering
		\includegraphics[scale=1.3]{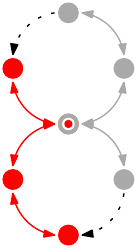}
		\label{sub:a}
		\caption{$\R_1$}
	\end{subfigure}
	\hfill
	\begin{subfigure}[b]{0.3\textwidth}
		\centering
		\includegraphics[scale=1.3]{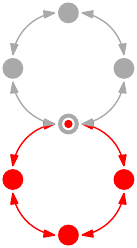}
		\label{sub:c}
		\caption{$\R_{12}$}
	\end{subfigure}
	\hfill
	\begin{subfigure}[b]{0.3\textwidth}
		\centering
		\includegraphics[scale=1.3]{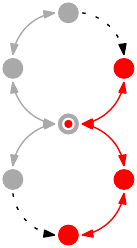}
		\label{sub:b}
		\caption{$\R_2$}
	\end{subfigure}
	\caption{Reducible representation $\Vt$ when $\t$  belongs to the indicated $\R_\alpha$. 
	}\label{fig:subs}
\end{figure}

\hspace*{-3em}\begin{minipage}{.7\linewidth}\centering
		\captionof{table}{Correspondence between $\t\in\R_I$ and the subrepresentation generated by $v_l\in\Vt$.
	}\label{table:corresp}
	\begin{tabular}{c|c|c}
		Subset & Highest Wt. Vector & Dimension\\\hline
		
		$\R_\emptyset$ & $E^{(111)}v_l$ & 8\\
		$\R_1$ & $E^{(011)}v_l$ & 4\\
		$\R_2$ & $E^{(110)}v_l$ & 4\\
		$\R_{12}$ & $E^{(010)}v_l$ & 4\\
		$\R_{1,12}$ & $E^{(101)}v_l$ & 3\\
		$\R_{12,2}$ & $E^{(010)}v_l$ & 3\\
		$\R_{1,2}$ & $E^{(000)}v_l$ & 1\\
	\end{tabular}
\end{minipage}\hspace*{-5em} \begin{minipage}{.6\linewidth}\centering
\includegraphics[scale=1.5]{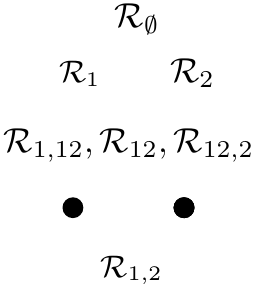}\captionof{figure}{Position of highest weight vector in the subrepresentation generated by $v_l\in\Vt$ for $\t\in \R_I$.}\label{fig:corresp}
\end{minipage}

\section{Projective Covers and Tensor Product Decompositions}\label{sec:proj}
This section is concerned with projectivity of $\Vt$ and showing the projective covers occur in tensor product decompositions. We show that $\Vt$ is projective in $\U$-mod if and only if ${\t\notin\R}$. If $\Vt$ is not projective, we prove that its projective cover is given by an induced representation. For the appropriate choices of $\t$, each of these projective covers are indecomposable. We show in Theorem \ref{thm:projdecomp} that $\Vt\otimes \Vs$ can be decomposed into a direct sum of these indecomposable projective representations. Aside from in Remark \ref{rem:proj}, we assume that $\g=\slthree$ throughout this section.
	
\projgen*

	\begin{proof}
		Suppose $\t\notin\R$. Let $N\twoheadrightarrow \Vt$ be a surjection, and let $n_0$ be a vector in the preimage of $v_h\in \Vt$. Then the map $v_h\mapsto E^{(111)}F^{(111)}n_0$ defines a splitting.\newline 
		
		Suppose $\t\in\R$ and $\mvec{\lambda}\in\P\setminus\R$. Consider the map $V(\mvec{-\lambda t})\hotimes V(\mvec{\lambda^{-1}}) \twoheadrightarrow\Vt$ defined by mapping $p_0:=v_h\hotimes v_l$ to $v_h$ and assume that a splitting exists.  Since $p_0$ is a cyclic vector for $V(\mvec{-\lambda t})\hotimes V(\mvec{\lambda^{-1}})$, the splitting can be assumed to take the form $v_h\mapsto \sum a_{\psi\psi'}F^\psi E^{\psi'}p_0$ for some $a_{\psi\psi'}\in\Q_4(t_1,t_2)$. Since $F^{(111)}v_h\neq0$, $a_{(000)(000)}=1$. Note that $E^{(111)}F^{(111)}v_h=0$ and $E^{(111)}F^{(111)}p_0=F^{(111)}E^{(111)}p_0+\sum_{\psi,\psi'<(111)} b_{\psi\psi'}F^{\psi}E^{\psi'}p_0$ for some $b_{\psi\psi'}\in\Q_4(t_1,t_2)$. In particular, $E^{(111)}F^{(111)}p_0\neq 0$, which is a contradiction. Thus, no splitting exists.
	\end{proof}

	\begin{rem}\label{rem:proj}
		Our method of proving Lemma \ref{lem:projgen} can be applied to $\Vt$ as a representation of $\Uzg$ for any semisimple Lie algebra $\g$. 
	\end{rem}

	Recall that a module $V$ is \emph{hollow} if every proper submodule $M\subseteq V$ is superfluous, i.e. for all $M\subsetneq V$ such that $V=M+N$ for some submodule $N$, then $N=V$. We say that $(P,p)$ is a \emph{projective cover} of $V$ if $P$ is projective, $p:P\to V$ is surjective, and $\ker p$ is a superfluous submodule of $P$.
	\\ 
	
	Let $B^\emptyset=B$, the Borel subalgebra, and for each nonempty $I\subsetneq \overline{\Phip}$ we set $B^I$ to be:
	\begin{align}
		B^{1}&=\bracks{K_1,K_2,E_{12},E_2},&
		B^{2}&=\bracks{K_1,K_2,E_1,E_{12}},\\
		B^{1,12}&=\bracks{K_1,K_2,E_{12}E_2,E_2},&
		B^{12,2}&=\bracks{K_1,K_2,E_1,E_1E_{12}},\\
		B^{12}&=\bracks{K_1,K_2,E_{12},F_2E_1,F_1E_2},&
		B^{1,2}&=\bracks{K_1,K_2,F_1E_2E_1,F_2E_1E_2}.
	\end{align}
	
	Let $\bracks{p_0^I}$ be the $B^I$-module on which the Cartan generators $K_i$ act by $t_i$ and all other generators of $B^I$ act trivially. Denote the induced representation $\text{Ind}_{B^I}^{\U}(\bracks{p_0^I})$ by $P^I(\t)$.
	 \projind*
	Let $\x=(-1,-\zeta)$ and $\y=(-\zeta,-1)$ be elements of $\P$, which are the multiplicative weight shifts under the actions of $E_1$ and $E_2$, respectively. We prove that each $P^I(\t)$ is hollow and projective for all $\t\in \R_I$, and as a corollary it is the projective cover of $\Vt$.\newline

We study each case individually, beginning with $I=\{\alpha_1\}$.

	\begin{prop}\label{prop:seq12,2}
		For any $\t\in\P$, $P^{1}(\t)$ belongs to the exact sequence
		\begin{align}
		0\to V(\x\t)\to P^{1}(\t)\to \Vt\to 0.
		\end{align}
		Moreover, this sequence splits if and only if $t_1^2\neq 1$. There is a similar exact sequence for $P^{2}(\t)$.
	\end{prop}
	\begin{proof}
		The representation $V(\x\t)$ is generated by the image of $E_1p_0^{1}$ under $U^-$ in $P^{1}(\t)$, and is a subrepresentation since $E_2E_1p_0^1=0$. Then quotient by $V(\x\t)$ is a surjection to $V(\t)$ since $E_2p_0^1=0$. \newline
		
		Accounting for weights, any splitting must be of the form $v_h\mapsto ap_0^{12,2}+bF_1E_1p_0^{1}$ for some $a,b\in \Q_4(t_1,t_2)$.  Since $E_1v_h=0$, $a=-b\floor{t_1}$. Composing the splitting with the surjection yields $v_h \mapsto -b\floor{t_1}v_h$, which can be normalized to the identity map if and only if $t_1^2\neq 1$.
	\end{proof}

	\begin{prop}\label{prop:proj1}
	Let $\t\in\R_1$, then $P^{1}(\t)$ is hollow and projective in $\U$-mod.
\end{prop}
\begin{proof}
	First, we show $P^{1}(\t)$ is projective. By construction, applying $E_2$ to either $p_0^{1}$ or $E_1p_0^{1}$ must be zero. Therefore, given a surjection $p:N\rightarrow P^{1}(\t)$ and any $n_0\in p^{-1}(p_0^{1})$, we claim that there is a splitting given by the map which sends $p_0^{1}\in P^{1}(\mvec{\sigma^\psi ts})$ to $E_{12}E_2F_{12}F_2n_0\in N.$	Expressing $E_{12}E_2F_{12}F_2n_0$ in the PBW basis, we see that its $n_0$ component is nonzero. Therefore, we may apply $F^\psi$ and $F^\psi E_1$ to this vector and obtain linearly independent vectors in $N$. Moreover, applying $E_2$ or $E_2E_1$ to this vector is zero. Since the weights of $p_0^{1}$ and $E_{12}E_2F_{12}F_2n_0$ agree, this map determines a splitting of any surjection onto $P(\t)$.\newline
	
	To prove that a proper submodule $M$ of $P^{1}(\t)$ is superfluous, suppose that $P^{1}(\t)=M+N$ for some $N\subseteq P^1(\t)$. Then there is a decomposition $p_0^{1}=(p_0^{1}-v)+v$ such that $p_0^{1}-v\in M$ and $v\in N$. Since $v$ is of weight $(t_1,t_2)$, $v=(a+bF_1E_1)p_0^{1}$ for some nonzero $a,b\in \Q_4(t_1,t_2)$. Then, $(F_1E_1)^2p_0^{1}=0$ implies $p_0^{1}=(\frac{1}{a}-\frac{b}{a}F_1E_1)v$ and $N=P^{1}(\t)$. 
\end{proof}
\begin{cor}
	For all $\t\in \R_1$, $P^{1}(\t)$ is the projective cover of $\Vt$.
\end{cor}

\comm{\begin{thm}\label{thm:proj1}
	Suppose $\t,\s\notin\R$ and either $\t\s$ or $\y\t\s$ belong to $\R_1$. Then $\Vt\otimes \Vs$ decomposes as a direct sum of indecomposable  representations:
	\begin{align}
	V(\t)\otimes V(\s)\cong \bigoplus_{\psi\in\Psi_1(\t\s)} P^{1}(\mvec{\sigma^\psi ts})\oplus \bigoplus_{\psi\in\Psi_\emptyset(\t\s)}V(\mvec{\sigma^\psi ts}).
	\end{align}
	We have a similar decomposition for $\t,\s\notin\R$ and either $\t\s$ or $\x\t\s$ belong to $\R_2$:
	\begin{align}
	\Vt\otimes \Vs\cong \bigoplus_{\psi\in\Psi_2(\t\s)}P^{2}(\mvec{\sigma^\psi ts})\oplus  \bigoplus_{\psi\in\Psi_\emptyset(\t\s)} V(\mvec{\sigma^\psi}\t\s).
	\end{align}
\end{thm}
\begin{proof}
	We describe an inclusion of the summands into the tensor product. For each $\psi\in\Psi_\emptyset(\t\s)$, we include the irreducible representation $V(\mvec{\sigma^\psi}\t\s)$ into $\Vts$ via the map $v_h\mapsto\Omega\hotimes F^\psi v_h$. Note that the image of each $V(\mvec{\sigma^\psi}\t\s)$ is distinct in $\Vts$. \newline 
	
	For each $\psi\in \Psi_1(\t\s)$, the map $p_0^1\mapsto E_{12}E_2F_{12}F_2v_h\otimes F^\psi v_h$ gives an inclusion of $P^1(\mvec{\sigma^\psi}\t\s)$ into $\Vts$. By another weight argument, these inclusions have distinct from each other and the image of each $V(\mvec{\sigma^\psi}\t\s)$. Together these inclusions span a 64-dimensional subspace of $\Vts$. Thus, yielding an isomorphism.
\end{proof}}
	We now consider the representation $P^{1,12}(\t)$, which is given by inducing over $B^{1,12}=\bracks{K_1,K_2,E_2,E_1E_{12}}. $
\begin{prop}\label{prop:seq2,112}
	For any $\t\in\P$, $P^{1,12}(\t)$ belongs to the exact sequence
	\begin{align}
	0\to P^{2}(\x\t)\to P^{1,12}(\t)\to \Vt\to 0
	\end{align}
	Moreover, this sequence splits if and only if $t_1^2\neq 1$.
\end{prop}
\begin{proof}
	Since $E_2E_1p_0^{1,12}\neq 0$, $E_1p_0^{1,12}$ generates a subrepresentation isomorphic to $P^{2}(\x\t)$. The quotient by this subrepresentation yields a surjection to  $\Vt$. \newline 
	
	The existence of a splitting uses the same argument as in Proposition \ref{prop:seq12,2}.
\end{proof}

\begin{prop}\label{prop:proj1i}
	For all $\t\in \R_{1,12}$, $P^{1,12}(\t)$ is a hollow projective in $\U$-mod.
\end{prop}
\begin{proof}
	To prove $P^{1,12}(\t)$ is projective, we modify the argument given in Proposition \ref{prop:proj1}. It can be shown that $p_0^{1,12}\mapsto E_2E_1F_1F_2n_0$ determines a splitting. \comm{ consider a surjection $N\twoheadrightarrow P^{1,12}(\t)$. Let $n_0$ be in the preimage of $p_0^{1,12}$. Since the $n_0$ component of $E_2E_1F_1F_2n_0$ is nonzero in the PBW basis and applying either $E_2$ or $E_1E_{12}$ to this vector is zero, the map} \newline
	
	We prove that any proper submodule $M$ of $P^{1,12}(\t)$ is superfluous. Suppose that $P^{1,12}(\t)=M+N$ for some $N\subseteq P$. We decompose $p_0^{1,12}$ into two vectors $p_0^{1,12}-v\in M$ and $v\in N$ with $v=Up_0^{1,12}$ for some $U\in \U$ such that $[K_i,U]=0$ for $i\in\{1,2\}$. Let $a,b,c,d\in \Q_4(t_1,t_2)$ so that $U=a+bF_1E_1+cF_{12}E_2E_1+dF_1F_2E_2E_1$. We may normalize $v$ so that $a=1$. We prove that $U$ is invertible by showing $U':=bF_1E_1+cF_{12}E_2E_1+dF_1F_2E_2E_1$ acts nilpotently on the $(t_1,t_2)$-weight space. We compute  $(U')^2=bdF_1F_2E_2E_1$, which implies $(U')^3=bdF_1F_2(E_2E_1U')=0$. Thus, $U^{-1}v=p_0^{1,12}$ and implies that $N=P^{1,12}(\t)$.
\end{proof}
\begin{cor}
	Let $\t\in \R_{1,12}$. Then $P^{1,12}(\t)$ is the projective cover of $\Vt$.
\end{cor}
\comm{\begin{thm}\label{thm:proj2112}
	Suppose $\t,\s\notin\R$, and either $\t\s\in\R_{1,12}$ or $\t\s\in\R_{12,2}$. Then we have a decomposition into indecomposable  representations:
	\begin{align}
	V(\t)\otimes V(\s)\cong \bigoplus_{\psi\in \Psi_{1,12}(\t\s)}P^{1,12}(\mvec{\sigma^\psi ts})\oplus \bigoplus_{\psi\in \Psi_{12,2}(\t\s)}P^{12,2}(\mvec{\sigma^\psi ts})\oplus \bigoplus_{\psi\in \Psi_\emptyset(\t\s)} V(\mvec{\sigma^\psi}\t\s).
	\end{align}
\end{thm}
\begin{proof}
	The inclusions of $P^{1,12}(\mvec{\sigma^\psi ts})$ and $P^{12,2}(\mvec{\sigma^\psi ts})$ into $\Vts$ are given by the maps $p_0^{1,12}\mapsto E_2E_1F_1F_2v_h\hotimes F^\psi v_h$ and $p_0^{12,2}\mapsto E_1E_2F_2F_1v_h\hotimes F^\psi v_h$, respectively. It can be verified that the images of these maps have trivial intersection. The other inclusions are given by the maps which send $v_h\in V(\mvec{\sigma^\psi ts})$ to $\Omega \otimes F^\psi v_h\in \Vts$.
	Since these inclusions have total dimension 64, this yields an isomorphism.
\end{proof}}
Next, we consider the subalgebra $B^{12}=\bracks{K_1,K_2,E_{12}, F_1E_2,F_2E_1}$. The reader can verify that the left $\U$-module generated by $B^{12}$ modulo the relations $\bracks{XK_i-Xt_i}$ for each $X\in \U$ is at most 48-dimensional. Therefore, $P^{12}(\t)$ is at least 16-dimensional. Using the following relations: $E_1E_2p_0^{12}=-\zeta E_2E_1p_0^{12}$, $\floor{\zeta t_1}E_2p_0^{12}= F_1E_1E_2 p_0^{12}$, and $\floor{\zeta t_2}E_1p_0^{12}=F_2E_2E_1 p_0^{12}$, we have that  $P^{12}(\t)$ is at most 16-dimensional. Therefore, $P^{12}(\t)$ is 16-dimensional and assuming $\floor{\zeta t_1}\neq0$ and $\floor{\zeta t_2}\neq0$, a PBW basis is given by
\begin{align}
\{F^\psi E^{\psi'}p_0^{1,2}:\psi\in\Psi, \psi'\in\{(000),(101)\}\}.
\end{align}

\begin{prop}\label{prop:seq12}
	For any $\t\in\R_{12}$, $P^{12}(\t)$ belongs to the exact sequences
	\begin{align}
	0\to V(\x\y\t)\to P^{12}(\t)\to \Vt\to 0.
	\end{align}
\end{prop}
\begin{proof}
	Since $E_{12}p_0^{12}=0$, then $E_1E_2p_0^{12}=-\zeta E_2E_1p_0^{12}$ and implies the inclusion of $V(\x\y\t)$ into $P^{12}(\t)$ by the map $v_h\mapsto E_1E_2p_0^{12}$ is well defined.\newline 
	
	We now verify that the quotient of $P^{12}(\t)$  by $V(\zeta t_1,\zeta t_2)$ is isomorphic to $\Vt$. Since the action of $U^-$ on the coset space $p_0^{12}+V(\x\y\t)$ is 8-dimensional, it is enough to verify that $E_1p_0^{12},E_2p_0^{12}\in V(\x\y\t)$.
	Since $\t\in\R_{12}$, both $\floor{\zeta t_1}$ and $\floor{\zeta t_2}$ are nonzero. Therefore, $E_1p_0^{12}=\dfrac{1}{\floor{\zeta t_1}} F_1E_1E_2 p_0^{12}$ and $E_2p_0^{12}=\dfrac{1}{\floor{\zeta t_2}} F_2E_2E_1 p_0^{12}$ vanish in the quotient $P^{12}(\t)/V(\x\y\t)$. 
\end{proof}
\begin{prop}
	Fix $\t\in\R_{12}$, then $P^{12}(\t)$ is a hollow projective in $\U$-mod.
\end{prop}
\begin{proof}
		To prove $P^{12}(\t)$ is projective, we modify the argument given in Proposition \ref{prop:proj1}. It can be shown that $p_0^{12}\mapsto E_{12}F_1F_2n_0$ determines a splitting.
\comm{	Let $N\twoheadrightarrow P^{12}(\t)$ be a surjection and suppose $n_0\in N$ is in the preimage of $p_0^{12}$. Then the map $p_0^{12}\mapsto E_{12}F_1F_2n_0$ determines a splitting since the $n_0$ component of $E_{12}F_1F_2n_0$ is nonzero when expressed in the PBW basis and applying $E_{12},E_1F_2,$ and $E_2F_1$ to this vector is zero. Thus, $ P^{12}(t)$ is projective.}\newline 
	
	We prove that $P^{12}(\t)$ is hollow. Suppose that $P^{12}(\t)= M + N$ for some representations $M$ and $N$. Then $p_0^{12}=Up_0^{12}+(1-U)p_0^{12}$ with $Up_0^{12}\in M$, $(1-U)p_0^{12}\in N$ and $U\in \U$ such that $K_iU=UK_i$ for $i\in \{1,2\}$. We may assume $U=1+aF_1E_1+bF_2E_2$ with $a,b\in \Q_4(t_1,t_2)$. We prove that at least one of $E_1E_2Up_0^{12}$ or $E_2E_1Up_0^{12}$ is nonzero, which implies $M$ is 16-dimensional and equals $P^{12}(\t)$. We compute
	\begin{align}
	E_1E_2(1+aF_1E_1+bF_2E_2)=(1-\zeta a\floor{\zeta t_1}-b\floor{t_2})E_1E_2\\
	E_2E_1(1+aF_1E_1+bF_2E_2)=(1-a\floor{t_1}+\zeta b \floor{\zeta t_2})E_2E_1.
	\end{align}
	If both of these expressions equal zero, then 
	\begin{align}
	\frac{1-b\floor{t_2}}{\zeta \floor{\zeta t_1}}=\frac{1+\zeta b\floor{t_2}}{\floor{t_1}}.
	\end{align}
	The above equality has no solution since $\t\in\R_{12}$. This proves the claim.
\end{proof}
\begin{cor}
	For all $\t\in\R_{12}$, $P^{12}(\t)$ is the projective cover of $\Vt$.
\end{cor}
\comm{\begin{thm}\label{thm:proj12}
	Suppose $\t,\s\notin\R$ and $\t\s$, $\x\t\s$, or $\y\t\s$ belong to $\R_{12}$. Then $\Vt\otimes \Vs$ decomposes as a direct sum of indecomposable  representations:
	\begin{align}
	V(\t)\otimes V(\s)\cong \bigoplus_{\psi\in\Psi_{12}(\t\s)}P^{12}(\mvec{\sigma^\psi ts})\oplus \bigoplus_{\psi\in\Psi_\emptyset(\t\s)}V(\mvec{\sigma^\psi ts}).
	\end{align}
\end{thm}
\begin{proof}
	The inclusion of $P^{12}(\mvec{\sigma^\psi ts})$ into $\Vts$ is given via the map $p_0^{12}\mapsto E_{12}F_1F_2v_h\hotimes F^\psi v_h$. The other inclusions are given by the maps which send $v_h\in V(\mvec{\sigma^\psi ts})$ to $\Omega \otimes F^\psi v_h\in \Vts$.
	Since these inclusions have trial intersection and their total dimension is 64, this yields an isomorphism.
\end{proof}}

	The last case to consider is the representation $P^{1,2}(\t)$, which is given by induction over $B^{1,2}=\bracks{K_1,K_2,F_1E_2E_1,F_2E_1E_2}$. Similar to the subalgebra $B^{12}$, the left $\U$-module generated by $B^{1,2}$ modulo the relations $\bracks{XK_i-Xt_i}$ for each $X\in\U$ can be shown to be at most 16-dimensional. Thus, $P^{1,2}(\t)$  is at least 48-dimensional and has the following relations $\floor{\zeta t_2}E_1E_2p_0^{1,2}=F_2E_2E_1E_2p_0^{1,2}$ and $\floor{\zeta t_1}E_2E_1p_0^{1,2}=-t_1F_1E_1E_2E_1p_0^{1,2}$. From these relations, we can show that $P^{1,2}(\t)$ is 48-dimensional. Assuming that $\floor{\zeta t_1}$ and $\floor{\zeta t_2}$ are nonzero, we replace any expression $F^\psi E_1E_2p_0^{1,2}$ and $F^\psi E_2E_1p_0^{1,2}$ with $-\frac{1}{\floor{\zeta t_2}}F^\psi F_2E_2E_1E_2p_0^{1,2}$ and $-\frac{1}{\floor{\zeta t_1}}F^\psi F_1E_1 E_2E_1p_0^{1,2}$, respectively. Under these assumptions $P^{1,2}(\t)$ has a basis given by
	\begin{align}
	\{F^\psi E^{\psi'}p_0^{1,2}:\psi\in\Psi, \psi'\in\{(000),(100),(001),(110),(011),(111)\}\}.
	\end{align}
	
	\begin{prop}\label{prop:proj11}
		Let $\t\in \R_{1,2}$. Then $P^{1,2}(\t)$ is a hollow projective in $\U$-mod.
	\end{prop}
	\begin{proof}
		Let $p:N\twoheadrightarrow P^{1,2}(\t)$ be a surjection and fix any vector $n_0\in p^{-1}(p_0^{1,2})$. The map $p_0^{1,2}\mapsto (1-t_1t_2F_1F_2E_1E_2-t_1t_2F_2F_1E_2E_1)n_0$ is a splitting of $p$.\newline 
		
		Suppose $P^{1,2}(\t)=M+N$ for some $M\subsetneq$ and $N\subseteq P$. Then $p_0^{1,2}=(p_0^{1,2}-v)+v$ for some $v=Up_0^{1,2}\in N$ such that $U$ is weight preserving, and $p_0^{1,2}-v\in M$. We may assume for some $a_{\psi\psi'}\in\Q_4(t_1,t_2)$ and $a_{(000)(000)}=1$,
		
		\begin{align}
		U=\sum_{\substack{\psi,\psi'\in \Psi\setminus\{(101),(010)\} \\P(\psi)=P(\psi')}}a_{\psi\psi'}F^\psi E^{\psi'}.
		\end{align}

		Since $(F_i E_i)^2p_0^{1,2}=0$ for $i\in\{1,2\}$, each of $1-a_{\delta_i\delta_i}F_i E_i$ is invertible on the $(t_1,t_2)$-weight space. By the same reasoning, $1-a_{{(111)}{(111)}}F^{(111)} E^{(111)}$ is invertible. Therefore,
		\begin{align}
		(1-a_{{(111)}{(111)}}F^{(111)} E^{(111)})(1-a_{(001)(001)}F_2E_2)(1-a_{(100)(100)}F_1E_1)Up_0^{1,2}
		\end{align} is equal to	$\left(1+aF^{(110)}E^{(110)}+bF^{(011)}E^{(011)}\right)p_0^{1,2}$
		for some $a,b\in \Q_4{(t_1,t_2)}$. It can be verified that multiplying this quantity by $
		{1-aF^{(110)}E^{(110)}-bF^{(011)}E^{(011)}+(a-t_1t_2b)^2F^{(111)} E^{(111)}}
		$  yields $p_0^{1,2}\in N$. Thus, $P$ is hollow as $M$ is superfluous.
	\end{proof}

	\begin{cor}\label{cor:proj0} For all $\t\in\R_{1,2}$, $P^{1,2}(\t)$ is the projective cover of $\Vt$.
\end{cor}
\comm{\begin{thm}\label{thm:projzero}
Suppose $\t,\s\notin\R$ and $\t\s\in\R_{1,2}$. Then we have a decomposition into indecomposable representations:
\begin{align}
V(\t)\otimes V(\s)\cong P^{1,2}(\mvec{-ts} )\oplus \bigoplus_{\psi\in \Psi_\emptyset(\t\s)} V(\mvec{\sigma^\psi}\t\s).
\end{align}
\end{thm}
\begin{proof}
	The summands $V(\mvec{\sigma^\psi}\t\s)$ are included into $\Vts$ via the maps $v_h\mapsto \Omega \hotimes F_{12}v_h$ and $v_h\mapsto \Omega \hotimes F_1F_2v_h$. These maps are nonzero since $\t\s\in \R_{1,2}$ implies $\x\y\t\s\notin\R$. Then $P^{1,2}(-\t\s)$ is included via the map $p_0^0\mapsto (1-t_1t_2F_1F_2E_1E_2-t_1t_2F_2F_1E_2E_1)v_h\hotimes v_l$. Since these inclusions intersect trivially and their dimensions sum to 64, they determine an isomorphism.
\end{proof}	}
 
 In Table \ref{table:index}, we define a collection of subsets $\mathscr{I}(J)$ dependent on $J\subseteq\overline{\Phip}$. For each $I\in \mathscr{I}(\t)$, we define 
 \begin{align}
 	\Psi_{I}(\t)&=\{\psi\in\Psi: \mvec{\sigma^\psi}\t\in\R_I\text{ and if } I\neq\emptyset, \exists\alpha\in I \text{ such that } F_\alpha F^\psi v_h=0\in \Vt \}.
 \end{align} 

These sets are characterized in the following table.

\begin{table}[h!]\caption{Description of $\mathscr{I}(J)$ and the respective number of elements in $\Psi_{I}(\t)$ for $\t\in \R_J$.}\label{table:index}
	\begin{tabular}{c|c|c}
	$J$ & $\mathscr{I}(J)$ & $|\Psi_{I}(\t)|$\\\hline
	$\emptyset$&$\emptyset$&8\\
	$\{\alpha\}$&$\emptyset,\{\alpha\}$&4,2\\
	$\{1,12\}$&$\emptyset,\{1,12\}, \{12,2\}$&2,1,1\\
	$\{12,2\}$&$\emptyset,\{1,12\},\{12,2\}$&2,1,1\\
	$\{1,2\}$&$\emptyset,\{1,2\}$&2,1\\
\end{tabular}
\end{table}

\projdecomp*

\begin{proof}
	The summands $V(\mvec{\sigma^\psi}\t\s)$ are included into $\Vts$ via the maps $v_h\mapsto \Omega \hotimes F^\psi v_h$. As in the proof of Lemma \ref{lem:inclusions}, these maps are nonzero since $\mvec{\sigma^\psi}\t\s\in \R_\emptyset$.\\
	
	Including the projective representation $P^I(\mvec{\sigma^\psi}\t\s)$ into the tensor product is given by the map splitting map used to show projectivity by setting $n_0=v_h\hotimes F^\psi v_h$. For $P^{12,2}(\mvec{\sigma^\psi}\t\s)$, we use the inclusion $p_0^{12,2}\mapsto E_1E_2F_2F_1 v_h\hotimes F^\psi v_h$. All of these inclusions have trivial intersection, and it can be readily verified from the Table \ref{table:index} that the total dimension is 64. Thus, proving the isomorphisms. 
\end{proof}
\begin{rem}
	The $\psi\in\Psi$ which determine some inclusion of $P^I(\mvec{\sigma^\psi}\t\s)\not\cong V(\mvec{\sigma^\psi}\t\s)$ are $(111)$ and the ones appearing in Table \ref{table:corresp} and Figure \ref{fig:corresp}.
\end{rem}
	\section{Cyclicity in the Generic Case}\label{sec:cyclicgen}
	We begin this section by defining a homogeneous cyclic representation. One goal for the remainder of this paper is to describe which representations $V(\t)\otimes V(\s)$ are homogeneous cyclic and extend Corollary \ref{cor:transfer} to those representations. The second goal is to give a complete description of when equation (\ref{eq:transfer}) holds. Thus, classifying all tensor product representations in the $\slthree$ case. In this section, we establish the methods used to find homogeneous cyclic vectors. By noting how a representation fails to be cyclically generated, we sort tensor product representations into families on which equation (\ref{eq:transfer}) holds for some $\mvec{\lambda}\in\P$. Throughout this section, we assume $\g=\slthree$.
	\begin{defn}
		A \emph{homogeneous cyclic vector} for a $\U$-module $M$ is a weight vector $\tilde{v}\in M$ such that 
		$
		M=\bracks{F^\psi E^{\psi'}.\tilde{v}:\psi,\psi'\in\Psi }.
		$
		We say that $M$ is generated by $\tilde{v}$ and call $M$ a \emph{homogeneous cyclic representation}. 
	\end{defn}
	\begin{lem}\label{lem:cyclic}
		If $\Vt\otimes\Vs$ and $V(\mvec{w})\otimes V(\mvec{z})$ are both homogeneous cyclic representations, then they are isomorphic if and only if $\mvec{ts}=\mvec{wz}$.
	\end{lem}
	\begin{proof}
		Since  $\Vt\otimes\Vs$ and $V(\mvec{w})\otimes V(\mvec{z})$ are assumed to be homogeneous cyclic, the both have a generating vector. An isomorphism between these representations must send one cyclic vector $\tilde{v}$ to another cyclic vector $\tilde{w}$ of the same weight. Therefore, the weight $\mvec{-ts}$ of $F^{(111)}E^{(111)}\tilde{v}$ is equal to the weight $\mvec{-wz}$ of $F^{(111)}E^{(111)}\tilde{w}$.
	\end{proof}
	We introduce the notation $wt(\mvec{\lambda})$ to denote the $\mvec{\lambda}$ weight space of $\Vts$. Recall the basis $\B'$ of $\Vts$ as in (\ref{eq:basis2}), which will be used throughout the remainder of the paper. 	Since $K_1$ and $K_2$ are group-like, the weight of a vector $F^{\psi_1}v_h\hotimes F^{\psi_2}v_h\in \Vts$ only depends on $P(\psi_1)+P(\psi_2)$ mod 4 and can be determined from (\ref{eq:xipowersforsigma}).
		\begin{lem}\label{lem:multiplyroots2}
		For each $\alpha\in{\overline{\Phip}}$ and $\psi_1,\psi_2\in\Psi$ such that $\psi_1(\alpha)$ is nonzero, there exists $a\in\Q_4$ such that
		$
		F_\alpha F^{\psi_1} v_h\hotimes F^{\psi_2} v_h =F^{\psi_1+\delta_\alpha}v_h\hotimes F^{\psi_2} v_h\neq 0.
		$ For each $\alpha\in{\Deltap}$ and $\Psi$, there exist coefficients $b_{\psi},c_{\psi}\in\Q_4$ such that 
		\begin{align}\label{eq:EFxF}
		E_\alpha F^{\psi_1} v_h\hotimes F^{\psi_2} v_h=\sum_{P(\psi')=-P(\delta_\alpha)+P(\psi_2)}c_{\psi'}F^{\psi_1} v_h\hotimes F^{\psi'} v_h+\sum_{P(\psi')=-P(\delta_\alpha)+P(\psi_1)}c_{\psi'}F^{\psi'} v_h\hotimes F^{\psi_2} v_h.
		\end{align}
	\end{lem}
	Equation (\ref{eq:EFxF}) follows from the intermediate step
	\begin{align}
	E_\alpha F^{\psi_1} v_h\hotimes F^{\psi_2} v_h=
	F^{\psi_1}E_\alpha.( v_h\hotimes F^{\psi_2})+\sum_{P(\psi')=-P(\delta_\alpha)+P(\psi_1)}c_{\psi'}F^{\psi'} v_h\hotimes F^{\psi_2} v_h.
	\end{align}
	\begin{lem}\label{lem:cyclic-ts}
		A homogeneous cyclic vector $\tilde{v}$ for $\Vts$, if one exists, must belong to the $\mvec{-ts}$ weight space and have a nonzero $v_h\hotimes v_l$ component.
	\end{lem}
	\begin{proof}
		Suppose $\Vts$ admits a  homogeneous cyclic vector $\tilde{v}\in wt(\mvec{\lambda})$ for some $\mvec{\lambda}\in\P$. Since $E^{(111)} \tilde{v}\neq0$, there is a nonzero component $F^{\psi_1}v_h\hotimes F^{\psi_2}v_h$ of $\tilde{v}$ such that 
		$
		E^{(111)}.(F^{\psi_1}v_h\hotimes F^{\psi_2}v_h)
		$ is nonzero.
		In particular, by Lemma \ref{lem:multiplyroots2}, 
		$
		P(\psi_1)+P(\psi_2)\geq (22).
		$
		Since $F^{(111)}\tilde{v}\neq0$, $\tilde{v}$ has a nonzero component $v_h\hotimes F^{\psi_3}v_h$. Since 
		$0\neq P(\psi_3)=P(\psi_1)+P(\psi_2)$ mod $4$, we have $\psi_3={(111)}$.
		Thus, $\lambda=\mvec{-ts}$, and $v_h\hotimes v_l$ is a nonzero component of $\tilde{v}$.
	\end{proof}

	Let $\pi$ denote the projection in $\Vts$ to the subspace $\bracks{v_h\hotimes F^\psi v_h:\psi\in\Psi}$ and $\pi_\psi$ the projection to $\bracks{v_h\hotimes F^\psi v_h}$, both taken with respect to the basis $\mathscr{B}$'. Let $d_\psi$ denote the scalar part of the projection $\pi_\psi$. Then, for every $v\in\Vts$,
	$
	\pi_\psi(v)=d_\psi(v)v_h\hotimes F^\psi v_h.
	$
	\begin{defn}\label{defn:effective}
		Let $\psi\in\Psi$. A vector $v\in wt(\mvec{-ts})$ is \emph{effective at level $\psi$} if there exists $x\in \U$ such that $\pi_\psi(xv)\neq 0$, otherwise $v$ is \emph{not effective at level $\psi$}.
	\end{defn}
	
	Informally, a $\mvec{-ts}$ weight vector is effective at level $\psi$ if any vector in its image under $U^+$ has a nonzero $v_h\hotimes F^\psi v_h$ component. A vector which is effective for each $\psi$ is a cyclic vector.

	\begin{lem}\label{lem:effective}
		Let $\psi_1,\psi_2,\psi_3\in\Psi$ such that $rt(\psi_2)<rt(\psi_3)$. Then $F^{\psi_1}v_h\hotimes  F^{\psi_2} v_h$ is not effective at level $\psi_3$. 
	\end{lem}
	\begin{proof}
		We consider the actions of $U^+$ and $U^-$ on $F^{\psi_1}v_h\hotimes  F^{\psi_2} v_h$. Since $U^+$ does not belong to the Borel subalgebra, the actions of $F_1$ and $F_2$ are only on the first tensor factor. For some coefficients $a_\psi\in\Q_4$, we have
		\[
		F_i.(F^{\psi_1}v_h\hotimes  F^{\psi_2} v_h)
		=\sum_{P(\psi)=P(\psi^{\alpha_i})+P(\psi_{1})}a_{\psi}F^{\psi}v_h\hotimes F^{\psi_2} v_h.
		\] In addition, there exist coefficients $b_\psi,c_\psi,\in\Q_4[t_1^\pm,t_2^\pm,s_1^\pm, s_2^\pm]$ so that
		\[
		E_i.( F^{\psi_1}v_h\hotimes  F^{\psi_2} v_h)
		=\sum_{P(\psi)=-P(\psi^{\alpha_i})+P(\psi_2)}b_{\psi}F^{\psi_1} v_h\hotimes F^{\psi} v_h
		+\sum_{P(\psi)=-P(\psi^{\alpha_i})+P(\psi_1)}c_{\psi}F^{\psi} v_h\hotimes F^{\psi_2} v_h.
		\] 
		In either case, each nonzero component of the resulting expression is a vector $F^\psi v_h\hotimes F^{\psi'} v_h $ with $rt(\psi')\leq rt(\psi_2)<rt(\psi_3)$. Thus, $F^{\psi_3}v_h$ cannot occur in the second tensor factor from the action of $U^+$ or $U^-$, and so $F^{\psi_1}v_h\hotimes  F^{\psi_2} v_h$ is not effective at level $\psi_3$. 
	\end{proof}
	
	The natural guess for a cyclic generator  $\tilde{v}\in\Vts$ is $v_h\hotimes  v_l$; however, it may not be the case that it generates the entire module. Indeed, multiplication by elements 
	
	$F^\psi E^{\psi'}$  on $v_h\hotimes v_l$ is given by
$
	F^\psi E^{\psi'}(v_h\hotimes v_l)=F^\psi (v_h\hotimes E^{\psi'} v_l)
	=(F^\psi v_h)\hotimes  (E^{\psi'} v_l).
	$ An instance of this can be seen in Example \ref{ex} and the expression vanishes if $s_1=\zeta$.\newline 

	Figure \ref{fig:standard} shows the subspace of $\Vts$ generated by $v_h\hotimes v_l$ under the action of $U^+$, assuming that it is a cyclic vector. To distinguish diagrams for $\Vts$ from those of $\Vt$, each vertex is labeled with a $\boxtimes$, and the multiplicity two weight space is labeled by $\otimes$. As before, each edge corresponds to a nonzero matrix element of either $E_1$ or $E_2$. However, in later diagrams, these may depend on the choice of generator $\tilde{v}\in wt(\mvec{-ts})$. We will assume $\tilde{v}$ is chosen maximally, in the sense that all possible nonzero matrix elements are present in the diagram. Since the action of $F_1$ and $F_2$ is  independent of the choice of parameters, we do not include edges corresponding to their action. 
	
	\begin{figure}[h!]
		\includegraphics[scale=1.5]{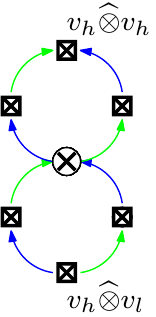}
		\caption{Diagram showing $v_h\hotimes  v_l$ as a homogeneous cyclic vector for the representation $\Vts$.}\label{fig:standard}
	\end{figure}
	
	\begin{prop}\label{prop:s degeneracy}
		The following are equivalent in $\Vts$:
		\begin{align}
			\bullet~v_h\hotimes  v_l \text{ is a homogeneous cyclic vector}
			&& \bullet~v_h\hotimes  v_l \text{ is effective at level } (000)
			&& \bullet~  	\s\notin \R.		\end{align} 	
	\end{prop}
	\begin{proof} 
		The first two statements are seen to be equivalent by considering 
		\begin{equation*}
		\{v_h\hotimes v_l,v_h\hotimes E^{(100)}v_l,v_h\hotimes E^{(010)}v_l,v_h\hotimes E^{(110)}v_l,v_h\hotimes E^{(001)}v_l ,v_h\hotimes E^{(101)}v_l,v_h\hotimes E^{(011)}v_l,v_h\hotimes E^{(111)} v_l\}.
		\end{equation*}
		As in Lemma \ref{lem:det}, this is a linearly independent set if and only if
		$
		E^{(111)}(v_h\hotimes  v_l)=v_h\hotimes  \Omega\neq0.
		$
		This is equivalent to $v_h\hotimes v_l$ being a cyclic generator. 
		The latter equivalence follows from Proposition \ref{prop:computation}, which shows that
		\[
		E^{(111)}(v_h\hotimes v_l)
		=v_h\hotimes  \Omega
		=-\zeta\floor{s_1}\floor{s_2}\floor{\zeta s_1s_2}v_h\otimes v_h.\qedhere
		\]
	\end{proof}
	We outline an informal algorithm which finds a homogeneous cyclic $\tilde{v}$ vector for  $\Vts$, if one exists, or tells one does not exist. This algorithm is a guide for the computations in Sections \ref{sec:cyclic1} and \ref{sec:cyclici}. 
	\begin{enumerate}
		\item Suppose $\tilde{v}=v_h\hotimes  v_l$.
		\item If $\tilde{v}$ is a generator stop, otherwise find the greatest $\underline{\psi}\in\Psi$ such that $\tilde{v}$ is not effective at level $\psi$.
		\item If there exists $v\in wt(\mvec{-ts})$ such that $\tilde{v}+v$ is effective at levels $\psi$ through ${(111)}$, then replace $\tilde{v}$ with $\tilde{v}+v$ and return to (2).  Otherwise, $\tilde{v}$ cannot be made into a generator and the representation is not homogeneous cyclic, stop.
	\end{enumerate}
	
	Proposition \ref{prop:s degeneracy} tells us to proceed to step $(3)$ of the algorithm if $\s$ belongs to $\X_1,\X_2$ or $\H$. Each case corresponds to different levels for which $v_h\hotimes v_l$ is not effective. Suppose a representation  $V(\t)\hotimes  V(\s)$ is not cyclic and the algorithm produces a vector $\tilde{v}$ which is not effective at level $\psi$. The diagram we obtain as a result is similar to the one in Figure \ref{fig:standard}, but with some edges, corresponding to the zero actions of $E_1$ and $E_2$, deleted.  In contrast to Figure \ref{fig:subs}, a disconnected graph implies a direct sum decomposition of the representation.  Moreover, $v_h\hotimes F^\psi v_h$ belongs to the head of $\Vts$. The head of $\Vts$ together with the product  $\mvec{ts}$ is enough to determine $\Vts$ up to isomorphism. We will be able to determine which algebraic sets contain $(\mvec{t},\mvec{s})$ from the diagrams we construct in the following sections, and therefore isomorphism classes of representations $\Vts$.
	\section{Cyclicity for  $\s\in\X_1\cup\X_2$}\label{sec:cyclic1}
	The cases with $\s\in\X_1\cup \X_2$ are easier to manage than those with $\s\in \H$, and so they are treated first. By the symmetry of the computations in this section, we only show the cases when  $\s\in\R_1$ and   $\s\in\R_{1,2}$. The conclusion of this section is given in Corollary \ref{cor:X1X2}. Throughout this section {we assume $s_1^{2}=1$ unless stated otherwise}. 
	
	\begin{lem}
		For $\tilde{v}$ to be effective at level $(011)$, the $F_1v_h\hotimes  F_{12}F_2v_h$ component of $\tilde{v}$ must be nonzero. Moreover, 
		$
		d_{(011)}(E_1.(F_1v_h\hotimes F_{12}F_2v_h))
		=s_1\floor{t_1}.
		$
	\end{lem}
	\begin{proof}
		We have already shown that $v_h\hotimes  v_l$ is not a homogeneous cyclic vector. More precisely, 
		\[
		E_1(v_h\hotimes  v_l)=v_h\hotimes  E_1v_l=-\floor{s_1}v_h\hotimes  F_{12}F_2v_h=0,
		\] 
		having referred to Table \ref{table:actions} and as $\floor{1}=\floor{-1}=0$. We wish to find a vector $\tilde{v}$ such that $E_1\tilde{v}$ is effective at level $(011)$. It follows from Lemma \ref{lem:effective} that $\tilde{v}$ must have a nonzero $F_1v_h\hotimes  F_{12}F_2v_h$ component. We apply $E_1$ to it, and by equation (\ref{eq:E1F32}), $E_1$ commutes with $F_{12}F_2$, 
		\[
		E_1.(F_1v_h\hotimes  F_{12}F_2v_h)=
		\floor{t_1s_1}v_h\hotimes  F_{12}F_2v_h=s_1 \floor{t_1}v_h\hotimes  F_{12}F_2v_h.\qedhere
		\]
	\end{proof}
	\begin{cor}\label{cor:011,110} 
		The pair $(\t,\s)$ is not effective at level $(011)$ if and only if it belongs to $\X_1^2$, and is not effective at level $(101)$ if and only if it belongs to $\X_2^2$.
	\end{cor} Let $\rho_B$ and $\rho_{U^{0}}$ be the projections in $\U$ to $B$ and $U^{0}$ in the PBW basis, respectively.
	
	\begin{lem} \label{lem:si is 1}
		Suppose $s_2$ is not a fourth root of unity. Then \begin{align}
		d_{(000)}(E_1E_{12}E_2.(v_h\hotimes v_l+F_1v_h\hotimes F_{12}F_2v_h))
		=\zeta \floor{s_2}\floor{\zeta s_2}\floor{t_1},
		\end{align} 
		and $ v_h\hotimes v_l+F_1v_h\hotimes  F_{12}F_2v_h$ is an effective vector for level $(000)$  if and only if $\t\notin \X_1$.
	\end{lem}
	\begin{proof}
		By equation (\ref{eq:E32F32}), 
		$
		\rho_{U^{0}} E_{12}E_2F_{12}F_2=-\zeta\floor{K_2}\floor{\zeta K_1K_2}
		$
		and by equation (\ref{eq:E132F1}),
		\begin{align*}
		E_1E_{12}E_2F_1v_h\hotimes  F_{12}F_2v_h&= E_{12}E_2\floor{K_1}v_h\hotimes  F_{12}F_2v_h=-\zeta \floor{s_2}\floor{\zeta s_2}\floor{t_1} v_h\otimes v_h.
		\end{align*}		
		Indeed, when $s_2$ is not a fourth root of unity, effectiveness of
		$
		F_1v_h\hotimes F_{12}F_2v_h
		$ at level $(000)$ is equivalent to $t_1^2\neq 1$.
	\end{proof}
	Similarly, if $s_1$ is not a fourth root of unity and $\s\in\X_2$, then
	$
	v_h\hotimes v_l+F_2v_h\hotimes F_1F_{12}v_h
	$
	is a homogeneous cyclic vector if and only if $\t\not\in\X_2$. Next, we suppose that both $s_1$ and $s_2$ are fourth roots of unity which square to 1. The case when $s_1$ and $s_2$ are fourth roots of unity and exactly one has square $-1$ is considered in the next section. 
	\begin{lem}\label{s1 s2 is 1}
		Let $\s\in\R_{1,2}$. The representation $\Vts$ is cyclic if and only if $\t\notin \mathcal{R}.$
	\end{lem}
	\begin{proof}
		By the above,
		$	v_h\hotimes v_l+F_1v_h\hotimes F_{12}F_2v_h+F_2v_h\hotimes F_1F_{12}v_h
		$ is not effective at level $(000)$ under the assumption $\mvec{s}\in\R_{1,2}$. Based on the previous computation, effectiveness only needs to be shown at level $(000)$ and only by appending $v_l\hotimes v_h$ may we obtain a cyclic vector.
		By Proposition \ref{prop:computation},
		\begin{align*}
		E^{(111)}.(v_l\hotimes  v_h)
		&=	\Omega \hotimes  v_h=
		-\zeta  \floor{t_1}\floor{t_2}\floor{\zeta t_1t_2}v_l\hotimes  v_h.\qedhere
		\end{align*}
	\end{proof}
	To summarize the results of this section, we have the following corollary. 
	\begin{cor}\label{cor:X1X2}
		If $(\t,\s)$ belongs to any of $\X_1^2$, $\X_2^2$, or 
		$\mathcal{R}\times\R_{1,2}$ then $V(\t)\hotimes V(\s)$
		is not cyclic. 
	\end{cor}
	Generically, each of these cases can be illustrated by omissions from the representation graph of Figure \ref{fig:standard}, they can be seen in Figure \ref{fig:si is 1}. 
	\begin{figure}[h!]
		\centering
	\begin{subfigure}[]{0.3\textwidth}
	\centering
	\includegraphics[scale=1.5]{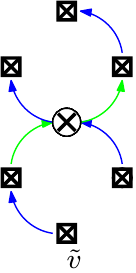}
	\caption{$\X_1^2$}
	\label{fig:1t1s}
\end{subfigure}	
		\hfill
		\begin{subfigure}[]{0.3\textwidth}
			\centering
			\includegraphics[scale=1.5]{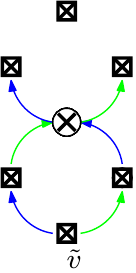}
			\caption{$\H\times\R_{1,2}$}
			\label{fig:tt11}
		\end{subfigure}
		\hfill
	\begin{subfigure}[]{0.3\textwidth}
\centering
\includegraphics[scale=1.5]{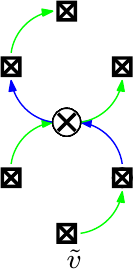}
\caption{$\X_2^2$}
\label{fig:t1s1}
\end{subfigure}
		\caption{Representation graph of $\Vts$ generated by $U^+$ acting on $\tilde{v}\in wt(\mvec{\lambda})$ when $(\t,\s)$ belongs to the indicated subset \hbox{of  $\P^2$.}
		}\label{fig:si is 1}
	\end{figure}
	
	\section{Cyclicity for  $\s\in\H$}\label{sec:cyclici}
	In addition to the cases considered in Section \ref{sec:cyclic1}, $v_h\hotimes  v_l$ is not a homogeneous cyclic vector when $(s_1s_2)^{2}=-1$. This section requires more work than the last because the computations involve the multiplicity two weight space occurring at levels $(101)$ and $(010)$. Throughout this section, {we assume that $\s\in\H$  unless specified otherwise.}
	\begin{lem}
		The vector subspace generated by $ E_1E_2v_h\hotimes  v_l$ and $E_2E_1v_h\hotimes  v_l$ has \hbox{dimension $1$.}
	\end{lem}
	\begin{proof}
		The proof is a computation of the vectors $E_1E_2v_l$ and $E_2E_1v_l$ into simplified terms which involve only the $F^\psi$, and showing they are multiples of each other. Using Table \ref{table:actions}, we express these vectors in the basis $\langle F^{(101)}v_h, F^{(010)}v_h \rangle$:
		\[E_1E_2v_l= \floor{s_2}E_1F^{(110)}v_h=\floor{s_2}(	\zeta 	s_1 F^{(101)}v_h-\floor{\zeta s_1} F^{(010)}v_h)=
		\floor{s_2}\begin{bmatrix}
		\zeta 	s_1\\- \floor{\zeta s_1}
		\end{bmatrix}\]
		and 
		\[
		E_2E_1v_l=\floor{s_1}E_2F^{(011)}v_h=\floor{s_1}(s_2^{-1}F^{(101)}v_h+\floor{s_2}F^{(010)}v_h) =
		\floor{s_1}\begin{bmatrix}
		s_2^{-1}\\\floor{s_2}
		\end{bmatrix}.\]
		The linear dependence is exhibited by computing the determinant of the matrix of coefficients, ignoring scale factors:
		\[
		\begin{vmatrix}
		{\zeta s_1}&{s_2}^{-1}\\ -\floor{\zeta s_1}&\floor{s_2}
		\end{vmatrix}=\floor{\zeta s_1s_2}=0. \qedhere
		\]\end{proof}
	\begin{cor}\label{cor:E101E011}
		The vectors $E^{(101)}v_h\hotimes  v_l$ and $E^{(011)}v_h\hotimes  v_l$ are zero when $\s\in\H$. Hence, 
		$v_h\hotimes  v_l$ is not effective at the levels $(001)$ and $(100)$.
	\end{cor}
	
	As in Lemmas \ref{lem:si is 1} and \ref{s1 s2 is 1}, we determine which $\mvec{-ts}$ weight vectors must be added to $v_h\hotimes v_l$ to yield a cyclic vector $\tilde{v}$ and state when no such vectors exist. We first produce a vector which is effective at either level $(101)$ or $(010)$, and does not belong to the $\langle E^{(101)}v_h\hotimes  v_l\rangle$ subspace.
	Recall that $\pi$ denotes the projection to the subspace $\bracks{v_h\hotimes F^\psi v_h:\psi\in\Psi}$ with respect to the basis $\mathscr{B}'$.\begin{defn}
		A vector ${v}\in wt{(\mvec{-ts})}$ is said to have the \emph{spanning property} if \begin{align}
		\text{span}\{ \pi E_1E_2v,\pi E_{12}v\}=\text{span}\{ v_h\hotimes  F_1F_2v_h,v_h\hotimes  F_{12} v_h\}.
		\end{align} 
	\end{defn}We consider four vectors: 
	\begin{align}
	F_1F_2v_h\hotimes  F_1F_2v_h, && F_1F_2v_h\hotimes  F_{12}v_h, &&
	F_{12}v_h\hotimes  F_1F_2v_h,  &&\text{and}&& F_{12}v_h\hotimes  F_{12}v_h
	\end{align}
	whose linear combinations are candidates for producing a vector with the spanning property. We denote the span of these vectors by $\Lambda$. \newline 
	
	Let $\Deltasl(\H)$ be the subset of $\H^2$ given by
	\begin{equation}\label{eq:H}
	\{(\t,\s)\in \H^2:(t_1s_1)^2=-1 \}=\{(\t,\s)\in \mathcal{P}^2 :(t_1s_1)^2= (t_2s_2)^2=(t_1t_2)^2=-1\}.
	\end{equation}
	Note that $\Deltasl(\H)$ is preserved by the action of $\mathcal{P}$ on $\mathcal{P}^2$ given in (\ref{eq:action}).
	
	\begin{lem}\label{b1b2 not cyclic}
		There exists a  vector $v\in\Lambda$ effective at either level $(101)$ or $(010)$ if and only if $(\t,\s)\notin \Deltasl(\H)$.
	\end{lem}
	\begin{proof} We show that effectiveness of all vectors $v\in\Lambda$ fails if and only if 	$t_1s_1\in\{\pm i\}$ and $t_2s_2\in \{\pm i\}$.
		We determine when the basis vectors of $\Lambda$ 
		are all simultaneously ineffective at levels $(101)$ and $(010)$. We focus on computing the relevant components of these vectors when acted on by $E_1E_2$ and $E_{12}$. Among the eight vectors to compute, we begin by computing $E_1E_2F_1F_2$, $E_1E_2F_{12}$, $E_{12}F_1F_2$, $E_{12}F_{12}$ and project them via $\rho_{U^{0}}$. We refer to equations (\ref{eq:E12F12}), (\ref{eq:E12F3}), (\ref{eq:E3F12}), and (\ref{eq:E3F3}) to obtain:
		\begin{align*}
		\rho_{U^{0}} E_1E_2F_1F_2&= \floor{K_1}\floor{K_2}, &
		\rho_{U^{0}} E_1E_2F_{12}&=
		- \floor{K_1}K_2^{-1},\\
		\rho_{U^{0}} E_{12}F_1F_2&=
		-\zeta K_1^{-1}\floor{K_2}, &
		\rho_{U^{0}} E_{12}F_{12}&=\floor{K_1K_2}.
		\end{align*}
		It follows that for $\psi\in \{(101), (010)\}$, we have:
		\begin{align*}
		d_\psi(E_1E_2.(F_1F_2v_h\hotimes  F^\psi v_h))&=\floor{\zeta t_1s_1}\floor{\zeta t_2s_2}, &
		d_\psi(E_1E_2.(F_{12}v_h\hotimes  F^\psi v_h))&=\zeta \floor{\zeta t_1s_1}(t_2s_2)^{-1},\\
		d_\psi(E_{12}.(F_1F_2v_h\hotimes  F^\psi v_h))&=(t_1s_1)^{-1}\floor{\zeta t_2s_2}, &
		d_\psi(E_{12}.(F_{12}v_h\hotimes  F^\psi v_h))&=	- \floor{t_1s_1t_2s_2}.
		\end{align*}
		The vanishing of $\floor{\zeta t_1s_1}$ and $\floor{\zeta t_2s_2}$ implies all the above expressions vanish. Thus, all of the above vectors vanish exactly when $(t_1s_1)^2=-1 $ and $ (t_2s_2)^2=-1$. In which case, all vectors in $\Lambda$ are ineffective for levels $(101)$ and $(010)$. Since $\mvec{s}\in\H$, we have proven the claim.
	\end{proof}
	
	\begin{lem}\label{linear indep}
		There exists $v\in \Lambda$ with the spanning property if and only if \begin{align}\label{eq:indep}
		(t_1s_1)^2\neq-1, && (t_2s_2)^2\neq -1, 
		&& \text{and} && (t_1t_2)^2\neq 1.  
		\end{align}
	\end{lem}
	\begin{proof}
		Let $v\in \Lambda$ be the linear combination with coefficients in $\Q[t_1^\pm,t_2^\pm,s_1^\pm,s_2^\pm]$,
		\[
		c_1F_1F_2v_h\hotimes  F_1F_2v_h + c_2F_1F_2v_h\hotimes  F_{12}v_h+
		c_3F_{12}v_h\hotimes  F_1F_2v_h  +c_4F_{12}v_h\hotimes  F_{12}v_h.
		\] 
		Let $v_{12}=\pi E_1E_2v$ and $v_{3}=\pi E_{12}v$. The $v_h\hotimes  F_1F_2v_h$ component of $v_{12}$ only comes from $c_1 F_1F_2v_h\hotimes  F_1F_2v_h$ and $c_3 F_{12}v_h\hotimes  F_1F_2v_h$, while its $v_h\hotimes  F_{12}v_h$ component only from $c_2 F_1F_2v_h\hotimes  F_{12}v_h$ and $c_4 F_{12}v_h\hotimes  F_{12}v_h$. Similarly for $v_{3}$. Moreover, each of these components have already been computed in the proof of Lemma \ref{b1b2 not cyclic}. In our current notation, we write $v_{12}$ and $v_{3}$ as vectors in the basis $\bracks{v_h\hotimes  F_1F_2v_h,v_h\hotimes  F_{12}v_h}$:
		\begin{align*}
		v_{12}=\floor{\zeta t_1s_1}\begin{bmatrix}
		c_1\floor{\zeta t_2s_2}+\zeta c_3(t_2s_2)^{-1} \\	
		c_2\floor{\zeta t_2s_2}+\zeta c_4(t_2s_2)^{-1}
		\end{bmatrix},
		&&
		v_{3}=\begin{bmatrix}c_1(t_1s_1)^{-1}\floor{\zeta t_2s_2}- c_3\floor{t_1s_1t_2s_2}\\
		c_2(t_1s_1)^{-1}\floor{\zeta t_2s_2}- c_4\floor{t_1s_1t_2s_2}\end{bmatrix}.
		\end{align*}
		We compute the determinant of the matrix whose columns are the vectors $v_{12}$ and $v_{3}$:
		\begin{align*}
		\begin{vmatrix}
		v_{12}& v_{3}
		\end{vmatrix}&=\begin{vmatrix}
		v_{12}& v_{3}+t_1s_1v_{12}
		\end{vmatrix}=-\zeta \floor{\zeta t_1s_1}\floor{\zeta t_1s_1t_2s_2}\floor{\zeta t_2s_2}(c_1c_4-c_2c_3).\qedhere
		\end{align*}
	\end{proof}
	The previous lemma has shown the spanning property independently of $v_h\hotimes  v_l$ being present. By adding other $\mvec{-ts}$ weight vectors to the vectors of $\Lambda$ just considered, the spanning property may hold more generally.
	We proceed by assuming at least one of the conditions in (\ref{eq:indep}) is not met.
	\begin{lem}\label{lem:t1s1}
		Suppose $(t_1s_1)^2=-1$. There exists  $v\in wt(\mvec{-ts})$ with the spanning property  if and only if 
		\begin{align}
		(\t,\s)\notin \Deltasl(\H)
		&&\text{ and either }&& \t\notin\X_2 && \text{ or } && \s\notin\R_{12,2}.
		\end{align}
	\end{lem}
	\begin{proof}
		By Lemma \ref{b1b2 not cyclic}, we require $(t_2s_2)^2\neq -1$ in order for a vector from the subspace $\Lambda$ to contribute to the spanning set. Let $v_{12}$ and $v_{3}$ be as in the proof of Lemma \ref{linear indep}. Under the present assumptions, $v_{12}$ is zero and under a relabeling 
		\[			v_{3}=\begin{bmatrix}c_1(t_1s_1)^{-1}\floor{\zeta t_2s_2}- c_3\floor{t_1s_1t_2s_2}\\
		c_2(t_1s_1)^{-1}\floor{\zeta t_2s_2}- c_4\floor{t_1s_1t_2s_2}\end{bmatrix}
		=\begin{bmatrix}c_1\floor{\zeta t_2s_2}- c_3\floor{t_1s_1t_2s_2}\\
		c_2\floor{\zeta t_2s_2}- c_4\floor{t_1s_1t_2s_2}\end{bmatrix}\]
		in the basis $\bracks{v_h\hotimes F_1F_2v_h,v_h\hotimes F_{12} v_h}$. \newline
		
		Together with $v_{3}$, only the vectors $E_1E_2.(v_h\hotimes  v_l)$ and $\pi E_1E_2.(F^{(010)}v_h\hotimes  F^{(101)}v_h)$ may contribute to the spanning set, as $\pi E_1E_2.(F^{(100)}v_h\hotimes  F^{(011))}v_h=0$.
		More explicitly, by equations (\ref{eq:E12F13}) and (\ref{eq:E1F13}), those vectors are
		\begin{align*}
		E_1E_2.(v_h\hotimes  v_l)
		&=v_h\hotimes (E_1E_2 F_1F_{12})F_2v_h
		=\floor{s_2}\begin{bmatrix}
		\zeta {s_1}\\- \floor{\zeta s_1}
		\end{bmatrix},
	\\
		\pi E_1E_2(F_2v_h\hotimes F_1F_{12}v_h)&=\pi E_1\floor{K_2}(v_h\hotimes F_1F_{12}v_h)=\floor{t_2s_2}\begin{bmatrix}
		\zeta s_1\\-\floor{\zeta s_1}
		\end{bmatrix}.
		\end{align*}
		Both of these vectors are zero when $s_2^2=1$ and $t_2^2=1$, in which case $\text{span}\{ \pi E_1E_2v,\pi E_{12}v\}$ is at most 1-dimensional.
	\end{proof}
	\begin{lem}\label{lem:1t1s}
		Suppose $(t_2s_2)^2=-1$. There exists  $v\in wt(\mvec{-ts})$ with the spanning property  if and only if 
		$
		(\t,\s)\notin \Deltasl(\H).
		$
	\end{lem}
	\begin{proof}
		By Lemma \ref{b1b2 not cyclic}, we require $(\t,\s)\notin \Deltasl(\H)$. However, for any such $(\t,\s)$ it can be shown that $E_{12}.(v_h\hotimes v_l)$ is non-zero and forms a spanning set with $v_{12}$.
	\end{proof}
	
	\begin{lem}
		Suppose $(t_1t_2)^2=1$.  There exists  $v\in wt(\mvec{-ts})$ with the spanning property  if and only if   $(\t,\s)\notin \R_{1,2}\times \R_{12,2}$.
	\end{lem}
	\begin{proof}
		Let $v_{12}$ and $v_{3}$ be as above so that
		\begin{align*}
		v_{12}=\floor{\zeta t_1s_1}\begin{bmatrix}
		c_1\floor{\zeta t_2s_2}+\zeta c_3(t_2s_2)^{-1}\\	
		c_2\floor{\zeta t_2s_2}+\zeta c_4(t_2s_2)^{-1}
		\end{bmatrix}
		&&
		\text{and}
		&&
		v_{3}=\begin{bmatrix}c_1(t_1s_1)^{-1}\floor{\zeta t_2s_2}- c_3t_1t_2\floor{s_1s_2}\\
		c_2(t_1s_1)^{-1}\floor{\zeta t_2s_2}- c_4t_1t_2\floor{s_1s_2}\end{bmatrix}
		\end{align*}	in the basis $\bracks{v_h\hotimes F_1F_2v_h,v_h\hotimes F_{12} v_h}$.
		Note that $v_{12}$ and $v_3$ are linearly {dependent}, by Lemma \ref{linear indep}. However, if $(t_1s_1)^2\neq-1$, then for any vector $w$ belonging to $\text{span}\{v_h\hotimes F_1F_2v_h,v_h\hotimes F_{12} v_h\}$ there are choices of $c_i$ so that $v_{12}$ equals to $w$ and similarly for $v_3$. In particular, $(t_1s_1)^2\neq-1$ implies $v_{12}$ and $E_{12}.(v_h\hotimes v_l)$ form a spanning set shown in the above proof. If  $(t_1s_1)^2=-1$, then $(t_2s_2)^2=1$. According to Lemma \ref{lem:t1s1},  we then require $s_2^2\neq 1$ in order to form a spanning set between $v_{3}$ and $E_1E_2.(v_h\hotimes v_l)$. Hence, there is a spanning set if and only if $(t_1s_1)^2\neq -1$ or $s_2^2\neq1$.
	\end{proof}
	
	\begin{cor}\label{cor:101,010}
		If $(\t,\s)$ belongs to $\Deltasl(\H)$ or $\X_2\times \R_{12,2}$ then 
		$(\t,\s)$ is not homogeneous cyclic.
	\end{cor}

	\begin{figure}[h!]
		\centering
		\begin{subfigure}[b]{.45\linewidth}
			\centering
			\includegraphics[scale=1.5]{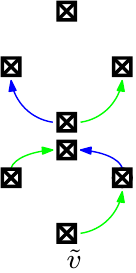}
			\caption{${\X_2\times \R_{12,2}}$}
			\label{fig:11i1}
		\end{subfigure}
		\hfill
		\begin{subfigure}[b]{.45\linewidth}
			\centering
			\includegraphics[scale=1.5]{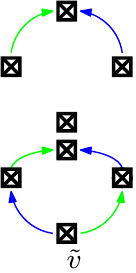}
			\caption{$\Deltasl(\H) $}
			\label{fig:titi}
		\end{subfigure}
		\caption{Representation graph of $\Vts$ generated by $U^+$ acting on $\tilde{v}\in wt(\mvec{\lambda})$ when $(\t,\s)$ belongs to the indicated subset \hbox{of  $\P^2$.}
		}\label{fig:(110) and (001)}
	\end{figure}
	
	Assuming that we have appended a vector $\tilde{v}$ which generates the subspace $\langle v_h\hotimes  F_1F_2v_h,v_h\hotimes  F_{12}v_h\rangle$, it remains to show that $\tilde{v}$ recovers the entire module. As noted in Corollary \ref{cor:E101E011}, \[E^{(110)}(v_h\hotimes  v_l)=E^{(011)}(v_h\hotimes  v_l)=0.\] As such we may neglect the $v_h\hotimes  v_l$ component of the cyclic vector at this point. In fact, the vectors $F_1v_h\hotimes  F_{12}F_2v_h$ and $F_2v_h\hotimes  F_1F_{12}v_h$ do not contribute to effectiveness beyond levels $(101)$ and $(010)$.
	We move our attention to effectiveness at levels $(001)$ and  $(100)$. We first determine whether $F_1F_{12}v_h\hotimes  F_2v_h$ and $F_{12}F_2v_h\hotimes  F_1v_h$ are effective before considering vectors in $\Lambda$. 
	
	\begin{lem}\label{F1F3 effective}
		The vector $F_1F_{12}v_h\hotimes  F_2v_h$ is effective at level $(001)$ if and only if \begin{align}
		(t_1s_1)^{2}\neq -1 &&\text{and} &&(t_1t_2)^{2}\neq -1,
		\end{align}
		and
		$F_{12}F_2v_h\hotimes  F_1v_h$ is effective at level $(100)$
		if and only if \begin{align}
		(t_2s_2)^{2}\neq -1 &&\text{and} &&(t_1t_2)^{2}\neq -1.
		\end{align}
	\end{lem}
	\begin{proof}
		First, we find $\rho_{U^{0}}E_1E_{12}F_1F_{12}$ given by equations (\ref{eq:E13F1}), (\ref{eq:E3F3}), and (\ref{eq:E12F3}),
		\begin{align*}
		\rho_{U^{0}}E_1E_{12}F_1F_{12}
		&=\rho_{U^{0}}((E_{12}\floor{\zeta K_1}+E_1E_2K_1^{-1})F_{12})=-\zeta \floor{K_1}\floor{\zeta K_1K_2}.
		\end{align*}
		Therefore, $d_{(001)}(E_1E_{12}.(F_1F_{12}v_h\hotimes  F_2v_h))=-\zeta \floor{\zeta t_1s_1}\floor{t_1s_1t_2s_2}$.
		Hence, $F_1F_{12}v_h\hotimes  F_2v_h$ is effective if and only if $(t_1s_1)^{2}\neq -1$ and $(t_1t_2)^{2}\neq -1$.
		A similar computation shows \[d_{(100)}(E_{12}E_2.( F_{12}F_2v_h\hotimes  F_1v_h))=-\zeta \floor{\zeta t_2s_2}\floor{t_1s_1t_2s_2}.\qedhere\] 
	\end{proof}
	
	\begin{lem}
		There does not exist a vector effective for levels $(100)$ and $(001)$ if and only if $(\t,\s)\in\H^2$.
	\end{lem}
	\begin{proof}
		After Lemma \ref{F1F3 effective}, it remains to compute the actions of $E_1E_{12}$ and $E_{12}E_2$ on $F^\psi v_h\hotimes  F^{\psi'} v_h$ for $\psi,\psi'\in \{(101),(010)\}$. We compute by equations (\ref{eq:E13F12}) and (\ref{eq:E13F3}):
		\begin{align*}
		\rho_B E_1E_{12}F_1F_2
		&=-\zeta E_1\floor{K_1K_2}, &
		\rho_B E_1E_{12}F_{12}&= E_1\floor{K_1K_2}.
		\end{align*}
		By equations (\ref{eq:E32F12}) and (\ref{eq:E32F3}),
		\begin{align*}
		\rho_B E_{12}E_2F_1F_2&=0 &
		\rho_B E_{12}E_2F_{12}&=\zeta E_2\floor{K_1K_2}.
		\end{align*}
		Observe that for each $\psi\in\{(101),(010)\}$,
		$
		\floor{K_1K_2}(v_h\hotimes  F^{\psi'} v_h)=
		-\floor{t_1s_1t_2s_2}v_h\hotimes  F^{\psi'} v_h.
		$
		So each action is zero exactly when $(t_1t_2)^{2}=-1$, in which case there is no effective vector in $\Lambda$. If $\t\in \H$, by Lemma \ref{F1F3 effective}, the vectors $F_1F_{12}\hotimes  F_2v_h$ and $F_{12}F_2v_h\hotimes  F_1v_h$ are also ineffective for levels $(100)$ and $(001)$. This proves the claim.
	\end{proof}

	\begin{cor}\label{cor:100,001}
		The representation $\Vts$ is not homogeneous cyclic due to ineffectiveness at levels $(100)$ and $(001)$ if and only if $(\t,\s)\in \H^2$.
	\end{cor}

\begin{figure}[h!]
		\centering
		\begin{subfigure}{\linewidth}
			\centering
			\includegraphics[scale=1.5]{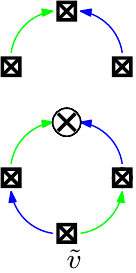}
			\caption{${\H^2}$}
			\label{fig:tisi}
		\end{subfigure}
		\caption{Representation graph of $\Vts$ generated by $U^+$ acting on $\tilde{v}\in wt(\mvec{\lambda})$ when $(\t,\s)$ belongs to the indicated subset \hbox{of  $\P^2$.}}
	\end{figure}
	Lastly, we investigate the $v_h\hotimes  v_h$ level by considering the action of $E^{{(111)}}$. 
	\begin{lem}\label{lem:vhvh}
		We have the following equalities:
		\begin{align}
		d_{(000)}(E^{{(111)}}.(F_1F_2v_h\hotimes  F_1F_2))&=-\floor{t_1s_1}\floor{s_2}\floor{\zeta s_1t_2s_2}\\
		d_{(000)}(E^{{(111)}}.(F_1F_2v_h\hotimes  F_{12}))&=s_1\floor{t_1s_1}\floor{t_2}\\
		d_{(000)}(E^{{(111)}}.(F_{12}v_h\hotimes  F_1F_2))&=-s_1s_2\floor{s_2}(\floor{s_1}\floor{\zeta t_1t_2}-\floor{t_1s_1}t_2^{-1})\\
		d_{(000)}(E^{{(111)}}.(F_{12}v_h\hotimes  F_{12}))&=s_1\floor{\zeta t_1t_2}\floor{s_1}-\zeta \floor{s_1s_2}\floor{t_1s_1}t_2^{-1}s_2^{-1}\\
		d_{(000)}(E^{{(111)}}.(F_1F_{12}v_h\hotimes  F_2))&=s_1s_2\floor{s_2}\floor{t_1t_2}\floor{t_1s_1}\\
		d_{(000)}(E^{{(111)}}.(F_{12}F_2v_h\hotimes  F_1))&=\zeta s_1s_2\floor{s_1}\floor{ t_1t_2}\floor{t_2s_2}.
		\end{align}
	\end{lem}
	\begin{proof}
		We compute by equations (\ref{eq:E132F12}), (\ref{eq:E132F3}), (\ref{eq:E132F1}), and (\ref{eq:E32F3}),
		\begin{align*}\allowdisplaybreaks
		\rho_B E_1E_{12}E_2F_1F_2
		&=-\zeta 
		E_1E_2\floor{\zeta K_1}K_2+ E_{12}\floor{\zeta K_1}\floor{K_2}
		\\
		\rho_B E_1E_{12}E_2F_{12}&= \zeta E_1E_2\floor{K_1K_2} - E_{12}\floor{\zeta K_1}K_2^{-1}
		\\
		\rho_B E_1E_{12}E_2F_1F_{12}
		&=- \zeta E_2\floor{K_1K_2}\floor{\zeta K_1}
		\\
		\rho_B E_1E_{12}E_2F_{12}F_2
		&= E_1 \floor{K_1K_2}\floor{\zeta K_2}.
		\end{align*}
		We include the first computation here, the others are similar:
		\begin{align*}\allowdisplaybreaks
		d_{(000)}(E^{{(111)}}.(F_1F_2v_h\hotimes F_1F_2v_h))&=
		(\floor{s_1}\floor{s_2})\floor{t_1s_1}(t_2s_2)-\zeta (s_1^{-1}\floor{s_2})\floor{t_1s_1}\floor{-\zeta t_2s_2}\\
		&=-\zeta s_1s_2\floor{t_1s_1}\floor{s_2}\floor{ t_2}.
		\qedhere
		\end{align*}	
	\end{proof}
	
	\begin{cor}\label{cor:000}
		The representation $\Vts$ is not effective for level $(000)$ if and only if $(\t,\s)$ belongs to any of $\X_1\times \R_{1,12}$, $\widehat{\Delta}(\H)$, $\X_2\times \R_{12,2}$, or $\R_{1,2}\times \H$.
	\end{cor}
	\begin{proof}
		Recall the underlying assumption that $(s_1s_2)^2=-1$. Proposition \ref{prop:computation} implies that $v_l\hotimes  v_h$ is not effective for level $(000)$ if and only if 
		\begin{align*}
		(	t_1s_1)^2=1,&&(t_2s_2)^2=1,&&\text{or} &&( t_1t_2)^2=1.
		\end{align*}
		We assume at least one such equality holds, otherwise $v_l\hotimes  v_h$ can be taken as a non-zero component of $\tilde{v}$ to produce a vector effective at level $(000)$. The vectors considered in Lemma \ref{lem:vhvh} may be used as a nonzero component of $\tilde{v}$. We determine when these vectors are all ineffective.
		\newline
		
		Suppose \emph{only} $(t_1s_1)^2=1$ then $F_{12}v_h\hotimes  F_1F_2v_h$, $F_{12}v_h\hotimes  F_{12}v_h$, and $F_{12}F_2v_h\hotimes  F_1v_h$ are effective at level $(000)$. Observe that $d_{(000)}E^{{(111)}}.(F_{12}v_h\hotimes  F_{12}v_h)=0$ only if $s_1^2=1$ or $(t_2s_2)^2=1$. We assume  $s_1^2=1$,  which implies $t_1^2=1$ and $s_2^2=-1$, and all vectors vanish. Hence, each $(\mvec{t},\mvec{s})\in \X_1\times \R_{1,12}$ is not effective at level $(000)$.
		\newline
		
		If $(t_1s_1)^2=1$ \emph{and} $(t_2s_2)^2=1$, then $(t_1t_2)^2=-1$ and all vectors are zero. Thus,  each $(\mvec{t},\mvec{s})\in\widehat{\Delta}(\H)$ is not effective at level $(000)$.
		\newline
		
		If we allow \emph{only} $(t_2s_2)^2=1$, then $d_{(000)}(E^{{(111)}}.(F_1F_2v_h\hotimes  F_{12}v_h))$ vanishes only if $t_2^2=1$. Thus, $s_2^2=1$ and $s_1^2=-1$. At this stage, all vectors vanish. Hence, a pair $(\mvec{t},\mvec{s})\in\X_2\times \R_{12,2}$ is not effective. \newline 
		
		So far we have not considered the $t_1t_2$. Thus, we assume only $(t_1t_2)^2=1$. Again, $F_1F_2v_h\hotimes  F_{12}v_h$ vanishes only if $t_2^2=1$. Assuming this, then $t_1^2=1$ and all vectors vanish.	Therefore, if $(\mvec{t},\mvec{s})\in\R_{1,2}\times \H$. then $\Vts$ does not have a vector effective at level $(000)$. This proves the claim.
	\end{proof}

	\begin{figure}[h!]
		\centering
		\begin{subfigure}{.3\linewidth}
			\centering
			\includegraphics[scale=1.5]{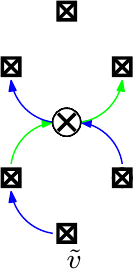}
			\caption{${\X_1\times \R_{1,12}}$}
			\label{fig:1t1i}
		\end{subfigure}
		\begin{subfigure}{.3\linewidth}
			\centering
			\includegraphics[scale=1.5]{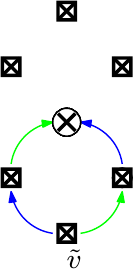}
			\caption{$\widehat{\Delta}(\H)$}
			\label{fig:titinvi}
		\end{subfigure}
		\begin{subfigure}{.3\linewidth}
			\centering
			\includegraphics[scale=1.5]{ti11VtVt}
			\caption{${\R_{1,2}\times \H}$}
			\label{fig:ti11}
		\end{subfigure}	
		\caption{Representation graph of $\Vts$ generated by $U^+$ acting on $\tilde{v}\in wt(\mvec{\lambda})$ when $(\t,\s)$ belongs to the indicated subset \hbox{of  $\P^2$.}
		}\label{fig:(000)}
	\end{figure}

	\section{The Cyclicity Theorem and Transfer Principle}\label{sec:thms}
	We have considered each of the cases identified in Proposition \ref{prop:s degeneracy}. Gathering the results of Corollaries \ref{cor:011,110}, \ref{cor:X1X2}, \ref{cor:101,010},  \ref{cor:100,001},   and \ref{cor:000} we may concisely characterize the existence of a homogeneous cyclic vector and the transfer principle.
	\begin{defn}
		The \emph{acyclicity locus} $\mathcal{A}$ is defined to the subset of $\mathcal{P}^2$ for which $V(\mvec{t})\otimes V(\mvec{s})$ is not homogeneous cyclic. 
	\end{defn}

	\cyclic*
	
	\begin{defn}
		The \emph{cyclicity stratification} of $\mathcal{P}^{2}$ is defined by the filtration 
		\begin{align}
		\mathcal{P}^2_0\subset\mathcal{P}^2_1=\mathcal{A}\subset\mathcal{P}^2_2=\mathcal{P}^2,
		\end{align}
		with
		\begin{equation}
		\mathcal{P}^{2}_0=\R_{1,2}^2\cup\R_{1,12}^2\cup\R_{12,2}^2\cup \DH\cup((\R_{1,12}\cup\R_{12,2})\times \R_{1,2})\cup( \R_{1,2}\times (\R_{1,12}\cup\R_{12,2}))
		\end{equation}
	\end{defn}

	\begin{rem}
		The maximal irreducible subspace generated by some $\tilde{v}\in wt(\mvec{-ts})$ for 
		\begin{align}
		(\t,\s)\in \H^2\cup (\H\times \R_{1,2})\cup 
		(\R_{1,2}\times \H)
		\end{align}
		has two highest weight vectors as seen in Figures \ref{fig:tt11}, \ref{fig:titi}, and \ref{fig:ti11}.
	\end{rem}
	\begin{rem}
		Non-degenerate implies homogeneous cyclic, with a homogeneous cyclic vector given by summing appropriate $\mvec{-ts}$ weight vectors from each direct summand. 
	\end{rem}
	\begin{defn}
		A \emph{transfer} is an isomorphism of representations determined by the action of $\mvec{\lambda}$ on $(\t,\s)$,
		$
		V(\t)\otimes V(\s)\cong V(\mvec{\lambda t})\otimes V(\mvec{\lambda^{-1} s}).
		$
		A transfer is called \emph{trivial} if $\mvec{\lambda}\in\R_{1,2}$. If $\mvec{\lambda}=\mvec{t^{-1}s}$ acts on  $(\t,\s)$ then the transfer is called a \emph{swap}.
	\end{defn}
	We group the defining subsets of $\P^2_0$ and $\P^2_1$ so that they are preserved by swaps, and we refer to the resulting subsets as \emph{symmetrized}. That is to say, we identify \begin{align}
	(\R_{1,12}\times \R_{1,2})\cup( \R_{1,2}\times\R_{1,12}) && \text{and} && (\R_{12,2}\times \R_{1,2})\cup( \R_{1,2}\times\R_{12,2})
	\end{align} as two, rather than four, algebraic sets in order to be preserved under swaps.
	\begin{cor}\label{cor:cyclic1}
		If $(\t,\s)\in \mathcal{P}^{2}\setminus\mathcal{A}$ and its image under $\mvec{\lambda}$ also belongs to $\mathcal{P}^{2}\setminus\mathcal{A}$, then $\mvec{\lambda}$ determines a transfer.
	\end{cor}

	\transfer*
	\begin{proof}
		The $n=2$ case implies $V(\t)\hotimes  V(\s)$ is a homogeneous cyclic representation. This case is a restatement of  Lemma \ref{lem:cyclic} and Corollary \ref{cor:cyclic1}. \newline
		
		Suppose $n=1$. Figures \ref{fig:t1s1}, \ref{fig:1t1s}, \ref{fig:tt11}, \ref{fig:ti11}, and \ref{fig:tisi} show that
		$\X_1^2$,  $\X_2^2$,
		$\H^2$, and $(\H\times \R_{1,2})\cup 
		(\R_{1,2}\times \H)$ determine non-isomorphic representations. These representations are generated by two or three vectors whose weights are determined by $\mvec{-ts}$, the weight of $\tilde{v}$. Since $\mvec{-ts}$ and the representation diagram are invariant under some $\mvec{\lambda}\in\mathcal{P}$, such a $\mvec{\lambda}$ determines a transfer. Note that the representation diagrams for $\Deltasl(\H)$ are different from those of $\H^2$, see Figure \ref{fig:titi}, but $\Deltasl(\H)\subseteq\H^2$ is preserved by $\mvec{\lambda}$.  \newline
		
		In the $n=0$ case, we only need to consider trivial transfers and swaps on $\mathcal{P}_0^2$. However, the same argument applies.
	\end{proof}

	\appendix
	
	\section{Commutation Relations}\label{sec:comp}
	In this section we gather the general computations used throughout the paper.\allowdisplaybreaks 
	\begin{align}
		[E_1,F_{12}]&=\zeta F_2K_1\\
		[E_2,F_{12}]&=- F_1 K_2^{-1}\\
		[E_{12},F_{12}]&=\floor{K_1K_2}\label{eq:E3F3}\\
		[E_1,F_1F_{12}]&=\zeta F_1F_2K_1- F_{12}\floor{\zeta K_1}	\label{eq:E1F13}\\
		[E_2,F_1F_{12}]&=0\label{eq:E2F13}\\
		[E_1,F_{12}F_2]&=0\label{eq:E1F32}\\
		[E_1E_2,F_1F_2]&=\floor{K_1}\floor{K_2}+F_2E_2\floor{K_1}+F_1E_1\floor{K_2}\label{eq:E12F12}\\
		[E_1E_2,F_{12}]&=F_2E_2K_1- (F_1E_1+\floor{K_1})K_2^{-1}\label{eq:E12F3}\\
		[E_{12},F_1]&=E_2K_1^{-1} \label{eq:E3F1}\\ [E_{12},F_2]&=\zeta E_1K_2\\
		[E_{12},F_1F_2]&=\zeta F_1E_1K_2-\zeta (\floor{K_2}+F_2E_2)K_1^{-1}\label{eq:E3F12}\\
		\label{eq:E12F13}
		[E_1E_2,F_1F_{12}]
		&= F_1F_2E_2K_1-F_{12}E_2\floor{K_1}
		\\ \label{eq:E32F12}
		[E_{12}E_2,F_1F_2]
		&=-\zeta F_1E_1E_2K_2+F_1E_{12}\floor{K_2}\\
		\label{eq:E32F3}
		E_{12}E_2F_{12}&=E_{12}(F_{12}E_2- F_1K_2^{-1})\\
		&=F_{12}E_{12}E_2+\floor{ K_1K_2}E_2-(F_1E_{12}+ E_2K_1^{-1})K_2^{-1}\notag\\
		&=
		F_{12}E_{12}E_2+E_2\floor{\zeta K_1K_2}-F_1E_{12}K_2^{-1}-E_2K_1^{-1}K_2^{-1}\notag
		\\
		\label{eq:E13F1}
		[E_1E_{12},F_1]
		&= E_{12}\floor{\zeta K_1}+ E_1E_2K_1^{-1}\\
		\label{eq:E132F1}
		[E_1E_{12}E_2,F_1]&=E_{12}E_2\floor{K_1}\\
		\label{eq:E13F12}
		[E_1E_{12},F_1F_2]
		&=-F_2E_{12}\floor{K_1}-\zeta F_2E_1E_2K_1^{-1}-\zeta E_1\floor{K_1K_2}\\
		\label{eq:E13F3}
		[E_1E_{12},F_{12}]
		&=- F_2E_{12}K_1+ E_1\floor{K_1K_2}
		\\
		\label{eq:E132F12}
		[E_1E_{12}E_2,F_1F_2]
		&= F_2E_{12}E_2\floor{\zeta K_1}
		-\zeta  E_1E_2\floor{\zeta K_1} K_2\\&\phantom{=}+ E_{12}\floor{\zeta K_1}\floor{K_2}+F_1E_1E_{12}\floor{K_2}\notag
		\\
		\label{eq:E132F3}
		E_1E_{12}E_2F_{12}&=E_1(F_{12}E_{12}E_2+E_2\floor{\zeta K_1K_2}- F_1E_{12}K_2^{-1}-E_2K_1^{-1}K_2^{-1})\\
		&=F_{12}E_1E_{12}E_2+\zeta F_2E_{12}E_2K_1+E_1E_2\floor{\zeta K_1K_2}\notag\\&\phantom{=}-E_1E_2K_1^{-1}K_2^{-1}- F_1E_1E_{12}K_2^{-1}- E_{12}\floor{\zeta K_1}K_2^{-1}\notag
	\\\label{eq:E32F32}
		E_{12}E_2F_{12}F_2&=(F_{12}E_{12}E_2+E_2\floor{\zeta K_1K_2}-F_1E_{12}K_2^{-1}-E_2K_1^{-1}K_2^{-1})F_2\\
		&=F_{12}E_{12}(F_2E_2+\floor{K_2})+ E_2F_2\floor{K_1K_2}+F_1E_{12}F_2K_2^{-1}-\zeta E_2F_2K_1^{-1}K_2^{-1}\notag\\
		&=F_{12}F_2E_{12}E_2+\zeta F_{12}E_1E_2K_2+F_{12}E_{12}\floor{K_2}+ F_2E_2\floor{K_1K_2}+ \floor{K_2}\floor{K_1K_2}\notag\\
		&\phantom{=}+ F_1F_2E_{12}K_2^{-1}+\zeta F_1E_1-\zeta F_2E_2-\zeta \floor{K_2}K_1^{-1}K_2^{-1}\notag
		\end{align}
	\section{The $\text{Ind}$ Functor}\label{Ind}
	In this appendix we define a general induced module. This construction is used to study tensor products of $V(\mvec{t})$ for $\U$ and sets the foundation for proving Theorem \ref{thm:directsum}. 
	\begin{defn}
		Let $A$ be an algebra and $B\subseteq A$ a subalgebra. Define ${\text{Ind}}_B^A:B\text{-}\it{mod} \rightarrow A\text{-}\it{mod}$ the \emph{induction functor} on $B$-modules by
		\begin{align}
		M\mapsto {\text{Ind}}_B^A(M):=A\otimes_B M=A\otimes M/\langle ab\otimes m-a\otimes b.m: a\in A, b\in B, m\in M\rangle.
		\end{align}
		Then ${\text{Ind}}_B^A(M)$ is indeed an $A$-module, with action given by multiplication in the left tensor factor. On $B$-equivariant maps, ${\text{Ind}}_B^A$ produces an $A$-equivariant map:
		\begin{align}
		f\in \Hom_B(M,N)\mapsto {\text{Ind}}_B^A(f):=id_A\otimes f\in \Hom({\text{Ind}}_B^A(M),{\text{Ind}}_B^A(N)).
		\end{align} 
	\end{defn} 
	The $A$-equivariance of ${\text{Ind}}_B^A(f)$ is straightforward to verify.
	\newline
	
	Consider the $B$-modules $M$ and $N$, with $B$ a sub-bialgebra of a bialgebra $A$. There are two types of induced representations on the tensor product of $M$ and $N$, namely
	\begin{align}
	\text{Ind}_B^A(M)\otimes \text{Ind}_B^A(N)=(A\otimes_BM)\otimes(A\otimes_B N)
	\end{align}
	and
	\begin{align}
	\text{Ind}_B^A(M\otimes \text{Ind}_B^A(N))=A\otimes_{B}(M\otimes (A\otimes_{B}N)).
	\end{align}
	Since $\text{Ind}_B^A(M)\otimes \text{Ind}_B^A(N)$ is a tensor product of $A$-modules,  $A$ acts via the coproduct action. Whereas $A$ acts by left multiplication on $\text{Ind}_B^A(M\otimes \text{Ind}_B^A(N))$, only elements of $B$ pass to $M\otimes \text{Ind}_B^A(N)$ which then utilize the coproduct.
	\begin{lem} 
		Let $A$ be a Hopf algebra and $M$ an $A$-module. Define \begin{align}
		\theta:A&\otimes M\rightarrow A\otimes M\\
		a&\otimes m\mapsto a'\otimes S(a'')m, \notag
		\end{align}  under the implied summation convention.
		Then the map $\theta$ is an isomorphism with inverse $\theta^{-1}(a\otimes m)=a'\otimes a''m$.
	\end{lem}
	
	\begin{rem}
		Note that $\theta$ satisfies the following commutative diagram.
		\[\begin{tikzcd}
		A\otimes M \arrow{r}{\theta} \arrow[swap]{d}{L_{\Delta(x)}} & A\otimes M \arrow{d}{L_x\otimes id} \\%
		A\otimes M \arrow{r}{\theta}& A\otimes M
		\end{tikzcd}\]
		Here $L_{\Delta(x)}$ denotes left multiplication of $x'\otimes x''$ on the tensor product, and $L_x$ is left multiplication by $x$. 
	\end{rem}
	
	\begin{prop}\label{prop:iso} Let $A$ be a Hopf algebra, $B\subseteq A$ a subalgebra, and $M$ a $B$-module. Define
		\begin{align}
		\Theta:\emph{\text{Ind}}_B^A(M)\otimes \emph{\text{Ind}}_B^A(N)&\rightarrow \emph{\text{Ind}}_B^A(M \otimes \emph{\text{Ind}}_B^A(N))\\
		[a_1\otimes_B m]\otimes [a_2\otimes_B n]&\mapsto 
		[a_1'\otimes_B (m\otimes [S(a_1'')a_2\otimes_B n])],\notag
		\end{align} under the implied summation convention. Then $\Theta$ defines a natural isomorphism of $A$-modules with inverse
		\begin{align}
		\Theta^{-1}\left(	[a_1\otimes_B (m\otimes [a_2\otimes_B n])]\right)=
		[a_1'\otimes_B m]\otimes [a_1''a_2\otimes_B n].
		\end{align}
		
	\end{prop}
	
	\begin{proof} 
		It is left to the reader to check that $\Theta$ and $\Theta^{-1}$ are indeed inverses.
		Observe that 
		\begin{align*}
		\text{Ind}_B^A(M)\otimes \text{Ind}_B^A(N)\cong (A\otimes M)\otimes (A\otimes N)/R_1 
		\end{align*}
		with 
		\[
		R_1=\langle (a_1b_1\otimes m)\otimes (a_2b_2\otimes n)-
		(a_1\otimes b_1. m)\otimes (a_2\otimes b_2.n): a_1,a_2\in A, b_1,b_2\in B, m\in M, n\in N\rangle
		\]
		and
		\begin{align*}
		\text{Ind}_B^A(M\otimes \text{Ind}_B^A(N))\cong A\otimes (M\otimes (A\otimes N)/R_2
		\end{align*}
		with 
		\[
		R_2= \langle a_1b_1\otimes (m\otimes (a_2b_2\otimes n))-
		a_1\otimes (b'_1. m\otimes (b''_1a_2\otimes b_2.n)): a_1,a_2\in A, b_1,b_2\in B, m\in M, n\in N\rangle.
		\]
		We first prove well-definedness of $\Theta$: 
		\begin{align*}
		&\Theta([a_1b_1\otimes_B m]\otimes [a_2b_2\otimes_B n]-[a_1\otimes_B b_1. m]\otimes [a_2\otimes_B b_2.n])
		\\&=[(a_1b_1)'\otimes_B (m\otimes [S((a_1b_1)'')a_2b_2\otimes_B n])]-[a'_1\otimes_B (b_1. m\otimes [S(a''_1)a_2\otimes_B b_2.n])]
		\\&=[a_1'\otimes_B b_1'(m\otimes [(S(b''_1)S(a''_1)a_2) \otimes_B b_2n])]-[a'_1\otimes_B (b_1. m\otimes [S(a''_1)a_2\otimes_B b_2.n])]\\&=0
		\end{align*}
		and similarly for $\Theta^{-1}$.
		\comm{\begin{align*}
		&\Theta^{-1} ([a_1b_1\otimes_B (m\otimes [a_2b_2\otimes_B n])]-[a_1\otimes_B (b'_1. m\otimes [b''_1a_2\otimes_B b_2.n])])
		\\&=[(a_1b_1)'\otimes_B m]\otimes [(a_1b_1)''a_2b_2\otimes_B n)]-[a_1'\otimes_B b'_1. m]\otimes [(a_1b_1)''a_2\otimes_B b_2.n]
		\\&=0.
		\end{align*}}
		It now remains to show commutativity of the following diagram for any choice of $B$-equivariant maps $f:M\rightarrow M'$ and $g:N\rightarrow N'$.
		\[\begin{tikzcd}
		\text{Ind}_B^A(M)\otimes \text{Ind}_B^A(N) \arrow{r}{\Theta} \arrow[swap]{d}{\text{Ind}_B^A(f)\otimes \text{Ind}_B^A(g)} & \text{Ind}_B^A(M\otimes \text{Ind}_B^A(N)) \arrow{d}{\text{Ind}_B^A(f\otimes \text{Ind}_B^A(g))} \\%
		\text{Ind}_B^A(M')\otimes \text{Ind}_B^A(N')\arrow{r}{\Theta}& \text{Ind}_B^A(M'\otimes \text{Ind}_B^A(N'))
		\end{tikzcd}\]
		We compute \begin{align*}
		&\text{Ind}_B^A(f\otimes \text{Ind}_B^A(g))\circ \Theta( (a_1\otimes m) \otimes( a_2 \otimes n))
		=a_1'\otimes (f(m)\otimes S(a_1'')a_2\otimes g(n)),
		\end{align*}
		which agrees with
		\begin{align*}
		&\Theta\circ (\text{Ind}_B^A(f)\otimes \text{Ind}_B^A(g))( (a_1\otimes m) \otimes( a_2 \otimes n))=\Theta((a_1\otimes f(m)) \otimes( a_2 \otimes g(n))). \qedhere
		\end{align*} 
	\end{proof}

	\bibliographystyle{alpha}

\end{document}